\documentclass[12pt]{amsart}
\usepackage[utf8]{inputenc}
\usepackage{amsmath,amsthm,amssymb,url}
\usepackage[foot]{amsaddr}
\usepackage[margin = 2.5cm]{geometry}
\usepackage{cleveref}
\usepackage{soul}
\usepackage{comment}
\usepackage{tikz}
\usepackage[colorinlistoftodos]{todonotes}
\newcommand{\matjaz}[1]{\todo[size=\tiny,inline,color=yellow!30]{#1 \\ \hfill --- M.}}
\newcommand{\sara}[1]{\todo[size=\tiny,inline,color=blue!30]{#1 \\
\hfill --- S.}}

\newcommand{\N}{{\mathbb N}}
\newcommand{\R}{{\mathbb R}}

\numberwithin{equation}{section}
\newtheorem{theorem}{Theorem}[section]
\newtheorem{prop}[theorem]{Proposition}
\newtheorem{cor}[theorem]{Corollary}
\newtheorem{lemma}[theorem]{Lemma}
\newtheorem{open}[theorem]{Open problem}
\theoremstyle{definition}
\newtheorem{defin}[theorem]{Definition}

\newtheorem{remark}[theorem]{Remark}
\newtheorem{example}[theorem]{Example}

\newcommand{\floor}[1]{\lfloor {#1} \rfloor}
\newcommand{\ceiling}[1]{\lceil {#1} \rceil}
\newcommand{\longto}{\longrightarrow}
\newcommand{\ASM}{\operatorname{ASM}}
\newcommand{\CSM}{\operatorname{CSM}}
\newcommand{\MT}{\operatorname{MT}}
\newcommand{\id}{\operatorname{id}}
\newcommand{\quilt}{\operatorname{Quilts}}
\DeclareMathOperator{\rk}{quiltrank}
\newcommand{\rank}{\operatorname{rank}}

\newcommand{\Fullrk}{\mathcal{F}_{k \times n}(\mathbb{C})}

\newcommand{\given}{\, | \,}
\newcommand{\ber}{b}

\title[Generalized Rank Functions
and Quilts of Alternating Sign
Matrices]{Generalized Rank Functions \\
and Quilts of Alternating Sign Matrices}

\date{\today}
\author{Sara C. Billey$^{1}$}
\address{1.Department of Mathematics, University of Washington, Seattle, WA, US}
\email{billey@math.washington.edu}

\author{Matja\v z Konvalinka$^{2}$}
\address{2.Faculty of Mathematics and Physics, University of Ljubljana \& Institute for Mathematics, Physics and Mechanics, Ljubljana, Slovenia}
\email{matjaz.konvalinka@fmf.uni-lj.si}

\begin{document}

\begin{abstract}
In this paper, we present new objects, \emph{quilts of alternating
sign matrices} with respect to two given posets.
Quilts generalize several commonly used concepts in mathematics.  For
example, the rank function on submatrices of a matrix gives rise to a
quilt with respect to two Boolean lattices.  
When the two posets are chains, a quilt is equivalent to an
alternating sign matrix and its corresponding corner sum matrix.  Quilts
also generalize the monotone Boolean functions counted by the Dedekind
numbers.  Quilts form a distributive lattice with many beautiful
properties and contain many classical and well-known sublattices, such
as the lattice of matroids of a given rank and ground set.  While
enumerating quilts is hard in general, we prove two major enumerative
results, when one of the posets is an antichain and when one of them
is a chain. We also give some bounds for the number of quilts when one
poset is the Boolean lattice.
\end{abstract}

\maketitle

\section{Introduction} \label{sec:intro}

The \textit{rank} of a matrix is fundamental in mathematics, science,
and engineering.  The notion of rank can also be associated to graphs,
matroids, and partial orders.  One can refine the rank function to
submatrices, subgraphs, etc.\ as well to get a family of ranks to
associate to each object.  We observe that such families always follow
certain Boolean growth rules leading to the concept of a generalized
rank function, which we call a quilt.  The goal of this paper is
to consider families of generalized rank functions and their
connection with the well-studied alternating sign matrices (ASMs). We
present some applications and some related enumerative questions.

One motivating example of generalized rank functions comes from the
Bruhat decomposition of the general linear group, permutations, and
the geometry of flag manifolds.  Given a matrix $M \in GL_{n}$, let
$\rank_{M}(i,j)$ be the rank of the submatrix of $M$ on rows
$[i,n]=\{i,i+1,\dotsc ,n \} \subseteq [n]$ and columns
$[j]=\{1,2,\dotsc ,j \} \subseteq [n]$.  We call the $n \times n$
matrix of values $\rank_{M}(i,j)$ for $i,j\in [n]$ the
\textit{southwest rank table} of $M$.  The southwest rank tables of
all $n\times n$ invertible matrices can be classified by permutations
in the symmetric group $S_{n}$ and their associated permutation
matrices.  To find the permutation matrix to associate to $M$ with the
same southwest rank table, simply apply all possible elementary column
reduction moves from left to right and elementary row reductions from
bottom to top.  These moves preserve the southwest rank of the matrix
and, since $M$ is invertible, will eventually terminate with a single
nonzero entry in each row and column. Finally, rescale each entry to
be $1$ to obtain the associated permutation matrix.

Bruhat order on the symmetric group $S_{n}$ is the partial order given
by $v\leq w$ if and only if $\rank_{M(v)}(i,j)\leq \rank_{M(w)}(i,j)$ for all
$1\leq i,j\leq n$, where $M(v),M(w)$ are the permutation matrices corresponding
to $v,w$.  Bruhat order also determines the partial order on
permutations given by containment of Schubert varieties in the flag
manifold \cite[Ch.~9]{Fulton-book}.  One observes from the Hasse
diagram of $S_{3}$ that Bruhat order is not a lattice, as a partial
order.  In general, the \textit{Dedekind--MacNeille} completion of a
finite poset is a lattice which contains the poset and is isomorphic
to some subposet of any other lattice containing it.  Lascoux and
Sch\"utzenberger proved that the Dedekind--MacNeille completion of
Bruhat order to a lattice is given by a natural partial order on
square alternating sign matrices (ASMs)
\cite{LS-MacNeille,lascoux2008chern}.  The (square) alternating sign
matrices were defined by Mills, Robbins and Rumsey to be the $n \times
n$ matrices with entries from the set $\{-1,0,1 \}$ such that the
nonzero entries in each row and each column alternate between $1$ and
$-1$ and must both start and end with $1$
\cite{Mills-Robbins-Rumsey.ASMS,Robbins-Rumsey}. They arise in the
process of computing a determinant using Dodgson condensation.  By
definition, permutation matrices are examples of ASMs.

The partial order on $n \times n$ ASMs generalizing Bruhat order is
given by $A=(a_{ij})\leq B=(b_{ij})$ if and only if the corresponding
entries in the \textit{southwest corner sum matrices} (CSMs) are increasing,
\begin{equation}\label{eq:nw.sum.test}
\sum_{i\leq  i'\leq n, 1\leq j'\leq j}a_{i'j'} \leq \sum_{i\leq
i'\leq n, 1\leq j'\leq j}b_{i'j'} \quad \mbox{for all } 1 \leq i,j \leq n.
\end{equation}
One can show that the ASM poset is a lattice with meet and join given by
taking the entry-wise min and max in the corresponding corner sum
matrices, and every such  matrix does indeed correspond to an
ASM. See \Cref{sec:background} for more details. Since the corner sum
matrix of a permutation matrix is exactly its southwest rank table, Bruhat order
on $S_{n}$ embeds into the ASM lattice.

There are exactly $\prod_{j=0}^{n-1}(3j+1)!/(n+j)! $ alternating sign
matrices of size $n \times n$.  This result was conjectured by Mills,
Robbins and Rumsey and proved first by Zeilberger
\cite{zeilberger-ASM}, and further established independently by
Kuperberg and Fischer \cite{kuperberg-ASM,Fischer.2005}.  It was
notoriously difficult to prove the enumeration formula for square
ASMs, and no simple formula for the number of rectangular ASMs is
known. 

In this paper, we define new objects, \emph{quilts of alternating sign
matrices}, corresponding to two ranked partially ordered sets with greatest and least elements. They are a
generalization of rectangular alternating sign matrices and their
corner sum matrices.  For example, the southwest rank table of a
matrix in $GL_{n}$ corresponds to a quilt on two copies of the chain
$C_{n}$.  Also, if $B_{n}$ is the Boolean lattice of subsets of $[n]$
ordered by inclusion and $M$ is a $k \times n$ matrix of full rank (i.e., the rank is $\min\{k,n\}$), the function
$f_{M}:B_{k} \times B_{n} \longrightarrow \mathbb{N}$ given by setting
$f_{M}(I,J)$ to be the rank of the submatrix of $M$ in rows $I$ and
columns $J$ is a quilt of type $(B_{k}, B_{n})$.  See
\Cref{ex:all.minors.quil}.

Given the complexity of enumerating rectangular ASMs, it is surprising
that we are still able to say quite a bit about the enumeration of
quilts, especially when one of the two posets is an antichain (see \Cref{thm:antichain}) or a
chain. One of our main results is the following theorem, more precisely
stated as \Cref{thm:chainenumeration}. Here $S(P)$ is the set of standard quilts,
defined in Section~\ref{sec:chain}.

\begin{theorem}
 The number of quilts of type $(P,C_n)$ for $n \geq \rank P$ is a polynomial of degree $b(P) = \sum_{x \in P} \rank x$. Furthermore, the leading coefficient is $\frac{|S(P)|}{b(P)!}$.
\end{theorem}

The following is an easy consequence of the theorem and the
hook-length formula for shifted standard tableaux.  To the best of our
knowledge, this is a new observation for rectangular ASMs.

\begin{cor}
 The number of rectangular ASMs of size $k \times n$, for $n \geq k$,
 is a polynomial in $n$ of degree ${\binom{k+1}2}$ with leading
 coefficient $\prod_{i=0}^{k-1}\frac{ (2i)!}{(k + i)!}$.
 \end{cor}

In general, it is rare to find exact formulas for quilt enumeration.
We show the problem of counting quilts on two arbitrary posets is
\#P-complete by a reduction to the enumeration of antichains, see
Theorem~\ref{thm:complexity}.  The quilts form a distributive lattice
with a number of beautiful properties, see Sections~ \ref{sec:def}
and~\ref{sec:lattice}.  They generalize the notions of matroids and
flag matroids, as we will show in \Cref{sec:lattice}.  As a precursor
to defining quilts, we also define Dedekind maps of posets
generalizing the monotone Boolean functions on $B_{n}$ counted by the
Dedekind numbers.  These numbers also count the number of antichains
in $B_{n}$.
 
Our main application of quilts is to an embedding of a partial order
on Fubini words (equivalently Cayley permutations or ordered set
partitions) into the lattice of quilts of type $(C_{k}, B_{n})$.  This
partial order on Fubini words arose in the context of a generalization
of southwest rank tables for decomposing rectangular matrices based on
the geometry of spanning line configurations due to Pawlowski and
Rhoades \cite{PR}.

The literature on alternating sign matrices and matroids is vast, and
we are certain that quilts hide many riches way beyond the scope of
this paper.  The paper is structured as follows. In
Section~\ref{sec:background}, we give some standard definitions about
partially ordered sets, alternating sign matrices, and matroids. In
Section~\ref{sec:def}, we define our main objects, \emph{Dedekind
maps} and   \emph{quilts} of
alternating sign matrices of type $(P,Q)$ for $P$ and $Q$ finite
ranked posets with least and greatest elements, and we develop some of
their basic properties. When $Q$ is a chain, we can view a quilt as a
filling of the poset $P$ with interlacing sets generalizing the notion
of a monotone triangle, see \Cref{prop:interlacing}.  In
Section~\ref{sec:motivation}, we explain how our definition of quilts
was motivated via the medium roast order on Fubini words. In
Section~\ref{sec:lattice} we prove some interesting properties of the
quilt lattice.  For example, there is a natural bijection between
quilts of type $(P,Q)$ and $(Q,P)$. In Sections~\ref{sec:antichain}
and~\ref{sec:chain}, we present our main enumerative results. We show
that the number of quilts when $Q$ is an antichain with $j$ elements
(and added least and greatest elements) is exponential in $j$, and
that it is a polynomial in $n$ when $Q$ is a chain of rank $n$. The
results also give asymptotic formulas in both cases. In
Section~\ref{sec:boolean}, we give some bounds for the number of
quilts when $Q$ is a Boolean lattice. In Section~\ref{sec:final}, we
point out some possible future research directions. In Appendix A, we
give some of the more unwieldy and computationally intensive
results. In Appendix B, we give some terms of the newly identified
integer sequences related to quilts.

\section{Background} \label{sec:background}

In this section, we lay out some of the standard notation for
permutations, posets, and alternating sign matrices.  See \cite{ec1}
for more information.

\subsection{Posets and lattices}

Let us write $\N$ for the set of non-negative integers. For $n \in
\N$, let $[n]$ be the set $\{1,\ldots,n\}$; in particular, $[0]=\emptyset$. Recall that a partially
ordered set (poset) is a set $P$ with a reflexive, antisymmetric and
transitive relation $\leq$. We denote by $[x,y]$ the \emph{interval}
between $x$ and $y$ for $x \leq y$. We say that $y$ \emph{covers} $x$,
denoted $x \lessdot y$, if $x < y$ and there is no $z$ satisfying $x <
z < y$. A \emph{chain} is a set of comparable elements in $P$, and an
\emph{antichain} is a set of incomparable elements. A chain is
\emph{maximal} if it is not contained in a larger chain. A poset is
\emph{ranked} if all maximal chains have the same size.   A ranked poset $P$
with least element $\hat 0_P$ and greatest element $\hat 1_P$ has a
\emph{rank function}, which is a map $\rank \colon P \to \N$ satisfying $\rank
\hat 0_P = 0$ and $x \lessdot y \Rightarrow \rank y = \rank x + 1$. We
define $\rank P = \rank \hat 1_P$, and we write
$b(P) = \sum_{x \in P} \rank x$ for the \textit{sum of ranks} of $P$.
We omit the subscript in $\hat 0_P$ and $\hat 1_P$ if the poset is clear from the context.

\textbf{Every poset in this paper is finite, ranked, and has the least
element $\hat 0$ and the greatest element $\hat 1$.}  The elements
covering $\hat 0$ are called \emph{atoms}, and the elements covered by
$\hat 1$ are \emph{coatoms}.  We assume that $P$ comes
with a fixed total ordering of the elements that respects rank: the element
$\hat 0=\hat 0_P$ comes first, then the atoms in some order, then the
elements of rank $2$, etc.

A poset is a \emph{lattice} if every two elements $x,y$ have a unique
greatest common lower bound (meet) $x \wedge y$ and a unique least common upper
bound (join) $x \vee y$. A lattice is \emph{distributive} if $(x \vee y)
\wedge z = (x \wedge z) \vee (y \wedge z)$ and $(x \wedge y) \vee z =
(x \vee z) \wedge (y \vee z)$ for all $x,y,z$. Some important examples
of  lattices that we will consider are the following:
\begin{itemize}
  \item $C_n$ for $n \geq 0$ is the poset $\{0,1,\ldots,n\}$ with the
  usual order $\leq$ (the \emph{chain} of rank $n$);
 \item $A_2(j)$ for $j \geq 1$ is the poset with $\hat 0$, $\hat 1$, and $j$ other elements that are incomparable with each other (the \emph{antichain} of $j$ elements, with $\hat 0$ and $\hat 1$ added);
 \item $B_n$ for $n \geq 1$ is the poset of subsets of $[n]$ ordered
 by inclusion (the \emph{Boolean lattice} of rank $n$).
\end{itemize}
Both $B_{n}$ and $C_{n}$ are distributive lattices, but $A_{2}(j)$ is
not distributive for $j>2$.  Note that the subscript always denotes
the rank of the poset. See \Cref{fig:c.a.b} for examples.

\begin{figure}[h]
\begin{center}
 \begin{tikzpicture}

\node (Z0) at (-8,0) {$0$};
\draw (-8,0.25) -- (-8,0.75);
\node (Z1) at (-8,1) {$1$};
\draw (-8,1.25) -- (-8,1.75);
\node (Z2) at (-8,2) {$2$};
\draw (-8,2.25) -- (-8,2.75);
\node (Z3) at (-8,3) {$3$};
\node (Y0) at (-4,0.5) {$\hat 0$};
\draw (-4.25,0.75) -- (-4.75,1.25);
\draw (-4,0.75) -- (-4,1.25);
\draw (-3.75,0.75) -- (-3.25,1.25);
\node (Y1) at (-5,1.5) {$x_1$};
\node (Y2) at (-4,1.5) {$x_2$};
\node (Y3) at (-3,1.5) {$x_3$};
\draw (-4.75,1.75) -- (-4.25,2.25);
\draw (-4,1.75) -- (-4,2.25);
\draw (-3.25,1.75) -- (-3.75,2.25);
\node (Y4) at (-4,2.5) {$\hat 1$};
\node (X0) at (0,0) {$\emptyset$};
\draw (-0.25,0.25) -- (-0.75,0.75);\draw (0,0.25) -- (0,0.75);\draw (0.25,0.25) -- (0.75,0.75);
\node (X11) at (-1,1) {$1$}; \node (X12) at (0,1) {$2$}; \node (X13) at (1,1) {$3$};
\draw (-1,1.25) -- (-1,1.75);\draw (-0.75,1.25) -- (-0.25,1.75);\draw (-0.25,1.25) -- (-0.75,1.75);\draw (0.25,1.25) -- (0.75,1.75);\draw (0.75,1.25) -- (0.25,1.75);\draw (1,1.25) -- (1,1.75);
\node (X21) at (-1,2) {$12$}; \node (X12) at (0,2) {$13$}; \node (X13) at (1,2) {$23$};
\draw (-0.25,2.75) -- (-0.75,2.25);\draw (0,2.75) -- (0,2.25);\draw (0.25,2.75) -- (0.75,2.25);
\node (X3) at (0,3) {$123$};
\end{tikzpicture}
\end{center}
\caption{Hasse diagrams for $C_3$, $A_2(3)$ and $B_3$.}
\label{fig:c.a.b}
\end{figure}
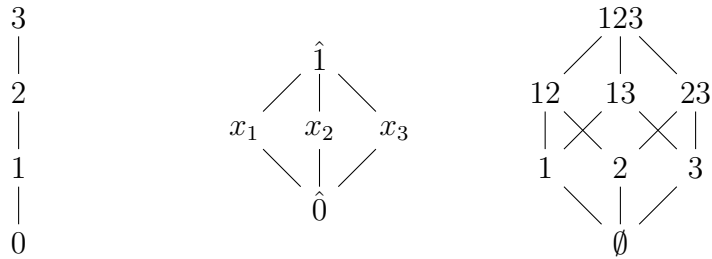

For posets $P$ and $Q$, the \textit{Cartesian product} $P \times Q$
has cover relations $(x,y) \lessdot (x',y)$ for $x \lessdot x'$ and
$(x,y) \lessdot (x,y')$ for $y \lessdot y'$.  For example, $B_{2}$ is
isomorphic to $C_{1}\times C_{1}$, and $B_{n}$ to $C_{1}^{n}$.  See
\cite[\S 3]{ec1} for more details on products and sums of posets.

\subsection{Alternating sign matrices and Fubini
words}\label{sub:asm.fubini}

An \emph{alternating sign matrix} (or ASM for short) is a matrix of
size $k \times n$ with entries in $\{-1,0,1\}$ such that:
\begin{itemize}
 \item in each row and each column the non-zero entries alternate,
 \item the leftmost non-zero entry in every row and the bottommost non-zero entry in every column is $1$,
 \item if $k \leq n$, there are no zero rows, and the rightmost non-zero entry in every row is $1$, and
 \item if $k \geq n$, there are no zero columns, and the topmost non-zero entry in every column is $1$.
\end{itemize}
 In particular, if $n = k$, the non-zero entries of every row and
 every column alternate and begin and end with $1$. Note that what we
 call an ASM is typically called a \emph{rectangular} or
 \emph{truncated} ASM in the literature; we will instead emphasize
 that we have a \emph{square} ASM when $n = k$. Denote the set of all
 ASMs of size $k \times n$ by $\ASM_{k,n}$. As mentioned in the
 introduction, a very famous result tells us that
 $$|\ASM_{n,n}| = \prod_{j=0}^{n-1} \frac{(3j+1)!}{(n+j)!}.$$
 See \cite{ASMs} for more background.

For example, every permutation matrix is a square ASM.  To spell out
some notation, a permutation in $S_{n}$ is a bijection $w:[n]
\longrightarrow [n]$, which can be denoted simply by its one-line
notation $w=w(1)w(2)\cdots w(n).$ The permutation matrix corresponding
with $w$ has a 1 in position $(w_{j},j)$ for each $j \in [n]$ and 0's
elsewhere.  More generally, a \textit{Fubini word} $w$ is a surjective
map $w:[n]\longrightarrow [k]$, denoted by its one-line notation
$w=w(1)w(2)\cdots w(n)$.  Let $W_{n,k}$ denote all such Fubini words
for a fixed pair $1\leq k\leq n$.  The Fubini words in $W_{n,k}$ are
in natural bijection with ordered set partitions on $[n]$ into $k$
nonempty parts given by $w^{-1}(1)|w^{-1}(2)|\dots | w^{-1}(k)$, but
in this context it is helpful to think of them as a $k \times n$
generalization of a permutation matrix where again the matrix for $w$
has a 1 in position $(w_{j},j)$ for each $j \in [n]$ and 0's elsewhere.
Such matrices are \emph{not} examples of rectangular ASMs of size $k
\times n$ when $k<n$ because some row must have two 1's with no $-1$
between them.

\begin{remark}
We note that the nomenclature ``Fubini words'' comes from
\cite[A000670]{oeis}.  Others refer to the same words as ``Cayley
permutations'', ``packed words'', ``surjective words'', ``normal
patterns'', and ``initial words.'' We thank Anders Cleasson for this
list.  Some history of this terminology is given in
\cite{claesson2024patternavoidingcayleypermutationscombinatorial}.
\end{remark}

The following is an example of an ASM of size $5 \times 6$:
\begin{equation}\label{eq:anotherASM}
\begin{bmatrix}
0 & 1 & 0 & -1 & 1 & 0 \\
1 & -1 & 0 & 1 & -1 & 1 \\
0 & 1  & 0 & -1 & 1 & 0 \\
0 & 0 & 0 & 1 & 0 & 0 \\
0 & 0 & 1 & 0 & 0 & 0
\end{bmatrix}.
\end{equation}
All ASMs of size $3 \times 3$ are
$$\begin{array}{ccccccc}
\begin{bmatrix} 1 & 0 & 0 \\ 0 & 1 & 0 \\ 0 & 0 & 1 \end{bmatrix} & \begin{bmatrix} 1 & 0 & 0 \\ 0 & 0 & 1 \\ 0 & 1 & 0 \end{bmatrix} & \begin{bmatrix} 0 & 1 & 0 \\ 1 & 0 & 0 \\ 0 & 0 & 1 \end{bmatrix} & \begin{bmatrix} 0 & 1 & 0 \\ 1 & \!-1\! & 1 \\ 0 & 1 & 0 \end{bmatrix} & \begin{bmatrix} 0 & 1 & 0 \\ 0 & 0 & 1 \\ 1 & 0 & 0 \end{bmatrix} & \begin{bmatrix} 0 & 0 & 1 \\ 1 & 0 & 0 \\ 0 & 1 & 0 \end{bmatrix} & \begin{bmatrix} 0 & 0 & 1 \\ 0 & 1 & 0 \\ 1 & 0 & 0 \end{bmatrix},
\end{array}$$
and we can get all ASMs of size $3 \times 2$ (resp., $2 \times 3$) by deleting the rightmost column (resp., the topmost row).

Note that reflecting across the vertical axis gives an involution on the set of ASMs of size $k \times n$ for $k \leq n$ (but not for $k > n$). When $n = k$, we have involutions coming from reflections across the horizontal and vertical axes and from the two diagonal transpositions of a matrix, as well as rotations by $90^\circ$, $180^\circ$ and $270^\circ$. This gives a faithful action of the dihedral group $D_4$ on the set of ASMs of size $n \times n$ for $n \geq 2$.

Given $A \in \ASM_{k,n}$, we can define a new matrix $C(A)$, called its
\emph{corner sum matrix}, or CSM for short, of size $(k+1) \times
(n+1)$ by adding a row and column of $0$'s below and to the left of
the matrix, and taking the sum of all entries weakly to the left and
weakly below a given entry. For example, the $5 \times 6$ ASM in
\eqref{eq:anotherASM} gives rise to the $6 \times 7$ CSM
\begin{equation}\label{eq:anotherCSM}
\begin{bmatrix}
0 & 1 & 2 & 3 & 3 & 4 & 5 \\
0 & 1 & 1 & 2 & 3 & 3 & 4 \\
0 & 0 & 1  & 2 & 2 & 3 & 3 \\
0 & 0 & 0 & 1 & 2 & 2 & 2 \\
0 & 0 & 0 & 1 & 1 & 1 & 1 \\
0 & 0 & 0 & 0 & 0 & 0 & 0
\end{bmatrix}.
\end{equation}

In the resulting matrix $C(A)$, the entries change by $0$ or $1$ when
moving to the right or up, and the bottom row and the leftmost column
always consist of $0$'s. Furthermore, if $k \geq n$, the top row
consists of $0,1,\ldots,n$, and if $k \leq n$, the rightmost column
consists of $0,1,\ldots,k$. Conversely, given a matrix $B$, with rows
numbered $0,1,\ldots,k$ and starting at the bottom, and columns
numbered $0,1,\ldots,n$ and starting on the left, satisfying these
properties, the $k \times n$ matrix $A=(a_{ij})$ given by
\[
a_{i,j} =b_{i,j}-b_{i-1,j}-b_{i,j-1}+b_{i-1,j-1}
\]
for all $i\in [k]$ and $j\in [n]$ is an ASM.  Therefore, we can
equivalently define CSMs directly as follows.

\begin{defin}\label{def:CSM}
A \emph{corner sum matrix (CSM)} of size $(k+1) \times (n+1)$ is a map
$$f \colon C_k \times C_n \longto \N$$
satisfying:
\begin{itemize}
 \item $f(i,j) = 0$ whenever $i =0$ or $j = 0 $,
 \item $f(k,n) = \min\{k,n\}$, and
 \item if $(i,j) \lessdot (i',j')$ in $C_k \times C_n$, then $f(i',j') \in \{f(i,j),f(i,j)+1\}$.
\end{itemize}
Let $\CSM_{k,n}$ denote the set of all CSMs of size $(k+1) \times (n+1)$.  We
refer to the third condition in the definition of a CSM as a
\textit{Boolean growth rule}, which is a central concept in this
paper.
\end{defin}

One more way to think of ASMs/CSMs is as monotone
triangles (MTs). Given a CSM $f \colon C_k \times C_n \longto \N$,
record the positions of \emph{jumps} in each row, denoted by
\begin{equation}\label{eq:jumps}
J_f(i) = \{j \in [n] \colon f(i,j) = f(i,j-1) + 1\}.
\end{equation}
Then, the \emph{monotone triangle} corresponding to $f$ is the
``triangular array'' of jump sequences with rows $J_{f}(k), \dotsc,
J_{f}(2), J_{f}(1)$ (from top to bottom).  The rows of the monotone triangle are
\emph{interlacing} in the sense that if $J_f(i)=\{s_{1}<s_{2}<\cdots <
s_{p} \}$ and $ J_f(i+1)= \{t_{1}<t_{2}<\cdots < t_{q}\}$ then either
$p= q - 1$ and
\begin{equation}\label{eq:interlacing.1}
t_1 \leq s_1 \leq t_2 \leq s_2 \leq \dots \leq s_{q-1}
\leq t_{q},
\end{equation}
or $p=q$ and
\begin{equation}\label{eq:interlacing.2}
t_1 \leq s_1 \leq t_2 \leq s_2 \leq \dots
\leq t_{q} \leq s_{q}.
\end{equation}
Clearly, the original CSM can be recovered
from its interlacing jump sets so there are easy bijections between
the ASMs, CSMs, and MTs for given $k,n$.

For the CSM in \eqref{eq:anotherCSM}, we get jump sequences $J_f(0) =
\emptyset$, $J_f(1) = \{3\}$, $J_f(2) = \{3,4\}$, $J_f(3) =
\{2,3,5\}$, $J_f(4) = \{1,3,4,6\}$, $J_f(5) = \{1,2,3,5,6\}$. This can
be presented in the monotone triangle
(MT) form as $$\begin{array}{ccccccccc} 1&&2&&3&&5&&6 \\ &1&&3&&4&&6& \\
&&2&&3&&5&& \\ &&&3&&4&&& \\ &&&&3.&&&& \end{array}
$$
Let $g$ be the CSM obtained by transposing \eqref{eq:anotherCSM} along the antidiagonal.
The jump sets are $J_g(0) = \emptyset$, $J_g(1) = \{4\}$, $J_g(2) =
\{3,5\}$, $J_g(3) = \{1,3,5\}$, $J_g(4) = \{1,2,4\}$, $J_g(5) =
\{1,2,3,5\}$, $J_g(6) = \{1,2,3,4,5\}$,
and the MT form
is $$\begin{array}{ccccccccc} 1&&2&&3&&4&&5 \\ &1&&2&&3&&5& \\
&&1&&2&&4&& \\ &&1&&3&&5&& \\ &&&3&&5&&& \\ &&&&4.&&&& \end{array}$$
Observe that $J_g(3)$ and $J_g(4)$ have the same size, so the MT form
is not a classical triangle of numbers.

\begin{remark}\label{rem:TerwilligerPoset}
Terwilliger introduced a poset $\Phi_{n}$ on the subsets of $[n]$ with
covering relations given by $S \lessdot T$ whenever $S$ and $T$ are
interlacing in the sense of \eqref{eq:interlacing.1}.  The poset
$\Phi_{n}$ contains the Boolean lattice $B_{n}$ as a subposet.  He
showed maximal chains in $\Phi_{n}$ are in bijection with $\ASM_{n,n}$,
just as the maximal chains of $B_{n}$ are in bijection with $S_{n}$
\cite[Thm. 3.4]{Terwilliger.2018}.  Building on this work, Hamaker and
Reiner \cite{hamakerreiner} showed that $\Phi_{n}$ is a shellable
poset, introduced a notion of descents for monotone triangles, and connected them to a
generalization of the Malvenuto--Reutenauer Hopf algebra of
permutations.  See \Cref{sub:generalizingASMs} for some follow up
questions in this direction.
\end{remark}

\subsection{Matroids and flag matroids}\label{sub:matroids}

A \emph{matroid} $M$ on ground set $[n]$ is determined by a rank
function $r:2^{[n]} \longrightarrow \mathbb{N}$ such that the following three conditions
hold:
\begin{enumerate}
\item (bounded by size) $0\leq r(X) \leq |X|$ for all $X \subseteq [n]$,
\item (monotonicity) $r(X)\leq r(Y)$ for all $X \subseteq Y$,
\item (submodularity) $r(X \cup Y) + r(X \cap Y) \leq r(X) + r(Y)$.
\end{enumerate}
The \emph{rank of $M$} is $r(M)=r([n])$.  For example, given a $k
\times n$ matrix with real entries, the rank function on the subsets
of columns of the matrix satisfies the three conditions above.

\begin{remark}\label{rem:matroid.dedekind.map}
Observe that if $i \in [n] \setminus X$, then $0\leq r(X \cup
\{i\})-r(X)$ by monotonicity.  By the bounded size property,
$r(\emptyset)=0$ and $r(\{i \})\leq 1$.  Hence, $r(X \cup
\{i\})-r(X)\leq r(\{i \})\leq 1$ by submodularity.  Hence, the rank
function of a matroid is surjective on $[0,r(M)]$ and also respects
the Boolean growth property: $r(X \cup \{i\}) \in \{r(X), r(X)+1 \}$.
\end{remark}

Let $M,N$ be matroids on $[n]$.  We say $N$ is a \emph{quotient} of
$M$ provided
\begin{equation}\label{eq:flag.matroid.rank.inequality}
r_M(Y)-r_M(X)\geq r_N(Y)-r_N(X)
\end{equation}
for all $X\subseteq Y \subseteq [n]$.  In particular, $r_M(Y)\geq r_N(Y)$ since
$r_M(\emptyset) = r_N(\emptyset)=0$.

A \emph{flag matroid} of type $(1\leq k_{1}<\cdots <
k_{s}\leq n)$ on $[n]$ is a collection of matroids
$\mathcal{M}=(M_{1}, M_{2},\dots , M_{s})$ on the ground set $[n]$
where the rank of $M_{j}$ is $k_{j}$ and $M_{j}$ is a quotient of
$M_{j+1}$ for each $1\leq j<s$.  For example, given a $k \times n$
matrix with complex entries, one can construct a flag matroid $\mathcal{M}=(M_{1}, M_{2},\dots , M_{s})$ where
$M_{i}$ for $1\leq i\leq k$ is the matroid with rank function defined by
the submatrix using only the top $i$ rows.

For more information on matroids and flag matroids, the standard
reference is the book by Oxley \cite{Oxley}.  A nice survey can be
found in \cite{CAMERON.DINU.MICHALEK.SEYNNAEVE}.

\section{Dedekind maps and quilts} \label{sec:def}

In this section, we introduce a generalization of the Dedekind
numbers, which count the number of monotone increasing Boolean
functions \cite[A000372]{oeis}.  Such functions are closely related to
the CSMs defined in \Cref{def:CSM} and are natural precursors to the
notion of a quilt defined later in this section.

\begin{defin}\label{def:Dedekin.map}
A \emph{Dedekind map of rank $k$} on $P$ is a surjective map $f \colon
P \to C_k$ satisfying $x \lessdot y \Rightarrow f(y) \in
\{f(x),f(x)+1\}$.  The set of all Dedekind maps of rank $k$ on $P$ is
denoted by $D_k(P)$, their union by $D(P)$, and we write $d_k(P) =
|D_k(P)|$ and $d(P) = |D(P)| = \sum_k d_k(P)$ for the \emph{$k$-Dedekind number
of $P$} and \emph{Dedekind number of $P$}, respectively.
\end{defin}

\begin{example}\label{ex:matroid.dedekind}
By \Cref{rem:matroid.dedekind.map}, the rank function of a matroid on
ground set $[n]$ of rank $k$ is a Dedekind map of rank $k$ on the
Boolean lattice $B_{n}$.  However, not every Dedekind map of rank $k$
is the rank function of a matroid of rank $k$.  In fact, any Dedekind
map on $B_{n}$ with $f(\{i\})=0$ for all $i\in [n]$ could not be the
rank function for a rank $k>0$ matroid on $[n]$ because submodularity
would be violated.  In general, there are far more Dedekind maps than
matroids. For example, there are only $7$ matroids on $[3]$ of rank
$2$, but there are $18$ Dedekind maps on $B_{3}$ of rank 2.
\end{example}

Observe that for a Dedekind map of rank $k$, $f \colon P \to C_k=[0,k]$, the following
three conditions are satisfied:
\begin{itemize}
\item $f(\hat 0) = 0$,
\item $f(\hat 1) = k$, and
\item if $x \lessdot y$, then $f(y) \in \{f(x),f(x)+1\}$ (Boolean growth).
\end{itemize}
Consequently, $f(x) \leq \rank x$ for all $x \in P$. The $k$-Dedekind
number of a chain is $d_k(C_n) = \binom n k$.  For the antichain poset
$A_{2}(j)$, we have
$$d_k(A_2(j)) = \left\{ \begin{array}{ccl} 1 & : & k = 0,2 \\ 2^j & : & k = 1 \\ 0 & : & k \geq 3. \end{array} \right. $$
Some values of $d_k(B_n)$ are given in the following table:

$$\begin{array}{c|cccccl}
   n \backslash k & 0 & 1 & 2 & 3 & 4 & 5  \\ \hline
   0 & 1 & 0 & 0 & 0 & 0 & 0 \\
   1 & 1 & 1 & 0 & 0 & 0 & 0 \\
   2 & 1 & 4 & 1 & 0 & 0 & 0 \\
   3 & 1 & 18 & 18 & 1 & 0 & 0 \\
   4 & 1 & 166 & 656 & 166 & 1 & 0 \\
   5 & 1 & 7579 & 189967 & 189967 & 7579 & 1.
  \end{array}$$
The first column determined by $d_1(B_n)$ for $n\geq 0$ is given by \cite[A007153]{oeis}.

\begin{remark}\label{rem:sharp.P}
Given $f \in D_1(P)$, the set of minimal elements satisfying $f(x) =
1$ is a non-empty antichain in $P \setminus \{\hat 0\}$; so $d_1(P)$
counts the number of antichains in $P$ (except for $\emptyset$ and
$\{\hat 0\}$), which is a \#P-complete problem \cite{antichains}.  See
also \cite{Sapozhenko} on antichain enumeration in ranked posets.

In particular, $d_1(B_n)+2$ is the classical Dedekind number and is
notoriously difficult to compute. The exact value for $n = 9$ was
first computed in 2023, thirty years after the value for $n=8$
\cite{JAKEL.2023}.  These numbers grow very quickly, $d_1(B_9) \approx
2.86 \cdot 10^{41}$.  See also \cite[A000372]{oeis} and
\cite{VanHirtum.dedekind9} for more history and an independent
computation of the $9$th Dedekind number.
\end{remark}

\begin{lemma} \label{lemma:dedekindupper}
 For any poset $P$ and $k \geq 1$, we have $d_k(P) \leq d_1(P)^k$.
\end{lemma}
\begin{proof}
 For $f \in d_k(P)$ and $1 \leq i \leq k$, take $A_i$ to be the set of minimal elements of the set $\{ x \in P \colon f(x) = i\}$. Clearly, $A_i$ is an antichain, $A_i \neq \emptyset$, $A_i \neq \{\hat 0\}$; there are $d_1(P)$ such antichains. The map $f \mapsto (A_1,\ldots,A_k)$ is an injection, which proves the statement.
\end{proof}

Every column of a CSM of size $(k+1) \times (n+1)$ can be seen as a Dedekind
map on $C_k$ and every row as a Dedekind map on $C_n$.  As one reads
left to right in columns or bottom to top in rows, another Boolean
growth rule must hold.  This second type of Boolean growth rule gives rise to the
following graphs.

\begin{defin}\label{def:Dedekind.graphs}
Let $G_D(P)$ denote the \emph{Dedekind graph of $P$}, defined as the
directed graph with vertex set given by the Dedekind maps in $D(P)$
and an edge from $f$ to $g$ if $g(x) \in \{f(x),f(x)+1\}$ for all $x
\in P$. The \emph{restricted Dedekind graph of $P$}, $G'_D(P)$, is the
directed graph with vertex set $D(P)$ and an edge from $f$ to $g$ if
$g(\hat 1_P) = f(\hat 1_P) + 1$ and $g(x) \in \{f(x),f(x)+1\}$ for all
$x \in P$.
\end{defin}

If we order the vertices $D(P) = \cup_{k=0}^{\rank P} D_k(P)$ by rank
$k$, and the vertices within $D_k(P)$ lexicographically (according to
our fixed linear order on $P$), the adjacency matrix $A_D(P)$ of
$G_D(P)$ is upper triangular with $1$'s on the diagonal. The adjacency
matrix $A'_D(P)$ of $G'_D(P)$ is strictly upper triangular.  In
particular, both $G_D(P)$ and $G'_D(P)$ are acyclic and have a unique
source and sink.  One can naturally identify the walks in the Dedekind
graph of a chain with CSMs, which leads to the next proposition.
Furthermore, we will use the (restricted) Dedekind graph of a poset to
prove the enumerative results in \Cref{thm:graph}.

\begin{figure}[!ht]
\centering
\includegraphics[height=3cm]{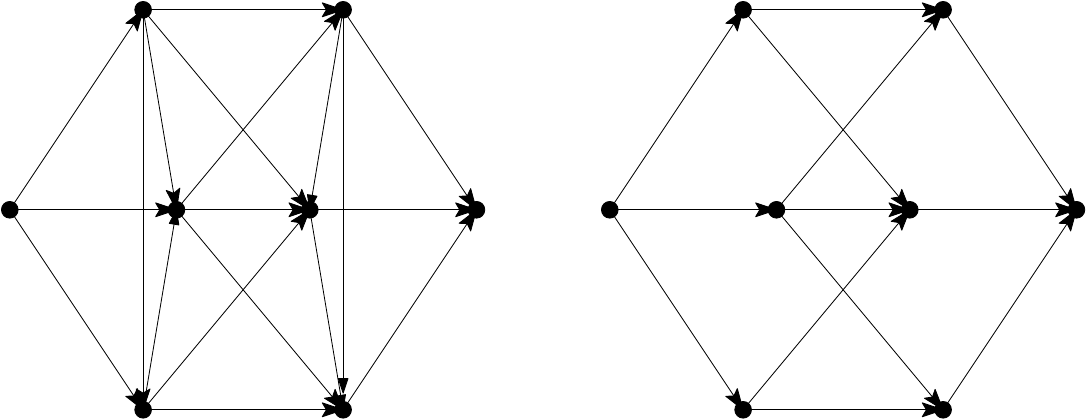}
\caption{The Dedekind graph (the loops are not shown) and the restricted Dedekind graph of $C_3$.}
\label{fig:dedekindgraph}
\end{figure}

\begin{prop}\label{prop:Dedekind.walks}
For any $1\leq k\leq n $, the map between the set $CSM_{k,n}$ and the
set of walks in the Dedekind graph $G_{D}(C_{n})$ from its unique sink
to a vertex in $D_{k}(C_{n})$ determined by the consecutive list of
columns is a bijection.
\end{prop}

\begin{proof}
The proof follows by comparing \Cref{def:CSM} and
\Cref{def:Dedekin.map}.
\end{proof}

Recall the definition of interlacing sets and monotone triangles from
\Cref{sec:background}.  The Dedekind graph of a chain $C_{n}$ also
respects the interlacing conditions.  The next statement follows from
\Cref{prop:Dedekind.walks}. See \Cref{prop:interlacing} for more
connections with the interlacing conditions.

\begin{cor}\label{rem:interlacing}
The restricted Dedekind graph $G'_{D}(C_{n})$ is isomorphic to the
directed graph on $B_n$, with an edge from $S$ to $T$ if $|S| = |T| -
1$ and the sets $S$ and $T$ are interlacing: $t_1 \leq s_1 \leq t_2
\leq s_2 \leq \ldots \leq s_{|T|-1} \leq t_{|T|}$. Similarly, the
Dedekind graph of $C_n$ is isomorphic to the directed graph on $B_n$,
edges as above plus an edge from $S$ to $T$ whenever $|S| = |T|$ and
$t_1 \leq s_1 \leq t_2 \leq s_2 \leq \ldots \leq t_{|T|} \leq
s_{|T|}$.
\end{cor}

\begin{remark}\label{rem:hasse.interlacing}
\Cref{rem:interlacing} implies that the Dedekind graph $G_{D}(C_{n})$
is the Hasse diagram of the Terwilliger poset $\Phi_{n}$ mentioned in
\Cref{rem:TerwilligerPoset}.  The restricted Dedekind graph
$G'_{D}(C_{n})$ is the Hasse diagram of the interlacing poset on the
subsets of $[n]$ with covering relations given by both types of
interlacing conditions.
\end{remark}

The following is the main definition of this paper.  It generalizes
the definition of a CSM in \Cref{def:CSM}.

\begin{defin} \label{def:quilt}
 Let $P$ and $Q$ be finite ranked posets with least and greatest
elements. A \emph{quilt of alternating sign matrices of type $(P,Q)$}
is a map $$f \colon P \times Q \longto \N$$ satisfying:
\begin{itemize}
 \item $f(x,y) = 0$ whenever $x = \hat 0_P$ or $y = \hat 0_Q$,
 \item $f(\hat 1_P, \hat 1_Q) = \min\{\rank P,\ \rank Q\}$, and
 \item if $(x,y) \lessdot (x',y')$ in $P \times Q$, then $f(x',y') \in
 \{f(x,y),f(x,y)+1\}$ (Boolean growth).
\end{itemize}
We will also call such a map an \emph{ASM quilt} or just a
\emph{quilt} for short. The set of all quilts of type $(P,Q)$ will be denoted by $\quilt(P,Q)$.
\end{defin}

\begin{remark}\label{rem:ASMquiltname}
A quilt of type $(C_k, C_n)$ is a CSM of size $(k+1) \times (n+1)$, so there
is also a corresponding ASM and MT. Similarly, for any $f \in
\quilt(P,Q)$ and any pair of maximal chains
$$\hat 0_P = x_0 \lessdot x_1 \lessdot \dots \lessdot x_{k-1} \lessdot
x_k = \hat 1_P, \quad \hat 0_Q = y_0 \lessdot y_1 \lessdot \dots
\lessdot y_{n-1} \lessdot y_n = \hat 1_Q$$ in $P$ and $Q$, the map
$(i,j) \mapsto f(x_i,y_j)$ is a CSM of size $(k+1) \times (n+1)$, which again
has a corresponding ASM and MT.  So we can think of quilts as encoding
collections of alternating sign matrices, one for each pair of maximal
chains in the two posets, appropriately ``pieced'' together like the
fabric of a quilt.
\end{remark}

\begin{example}\label{ex:all.minors.quil}
Let $M$ be a $k \times n$ matrix of full rank.  Take the function
$f_{M}:B_{k} \times B_{n} \longrightarrow \mathbb{N}$ given by setting
$f_{M}(I,J)$ to be the rank of the submatrix of $M$ in rows $I$ and
columns $J$; here $f_{M}(I,J)=0$ if either $I$ or $J$ is the empty
set.  Since $M$ is full rank, we know $f_{M}([k],[n])=\min\{k,n\}$.
Since the rank of any matrix increases by at most 1 when we add in one
more row or column, it is clear that the Boolean growth rule in
\Cref{def:quilt} is satisfied as well.  Hence, $f_{M}:B_{k} \times
B_{n} \longrightarrow \mathbb{N}$ is a quilt of type $(B_{k}, B_{n})$.
This example also proves rank functions of graphs and their subgraphs
are encoded by quilts since the rank of a graph can be defined as the
rank of its adjacency matrix.
\end{example}

\begin{example} \label{ex:dedekindtoquilt}
Take posets $P$ and $Q$ with $\rank P \leq \rank Q$, and an $l$-Dedekind map $g$ on $Q$, $l \geq \rank P$. Then the map
$$f \colon P \times Q \longrightarrow \N, \qquad f(x,y) = \min\{\rank
x,\ g(y)\}$$ is a quilt of type $(P,Q)$. In particular, if $g_1$ and
$g_2$ map to the same quilt $f$, then $g_1(y) = f(\hat 1_P,y) =
g_2(y)$ for all $y \in Q$, so $D_{\rank P}(Q) \longrightarrow
\quilt(P,Q)$ is injective.

\end{example}

\begin{lemma}\label{lem:chains.bounds}
Let $f \in \quilt(P,Q)$.  If $\rank P \geq \rank Q$, then $f(\hat 1_P,y) = \rank_{Q}
y$ for all $y \in Q$.  If $\rank P \leq \rank Q$, then $f(x, \hat 1_{Q}) = \rank_{P} x$ for all $x \in P$.
\end{lemma}

\begin{proof}
Assume $k = \rank P \geq n = \rank Q$.  Fix any maximal chain $\hat 0_Q = y_0 \lessdot
y_1 \lessdot \dots \lessdot y_{n-1} \lessdot y_n = \hat 1_Q$ in
$Q$. Then by definition of a quilt, we have $0 = f(\hat 1_P,y_0) \leq
f(\hat 1_P,y_1) \leq \dots \leq f(\hat 1_P,y_n) = \min\{k,n\} = n$ and $f(\hat
1_P,y_i) - f(\hat 1_P,y_{i-1}) \in \{0,1\}$ for $i = 1,\ldots,n$, so
the first claim follows.  The second statement follows similarly. 
\end{proof}

The lemma also shows that the entire rank function of the smaller ranked
poset is encoded in each quilt.  This justifies our claim that quilts
generalize rank functions of posets. 

\bigskip

 There is a natural partial order on $\quilt(P,Q)$: we say that $f \leq g$ if $f(x,y) \leq g(x,y)$ for all $x \in P$, $y \in Q$. For $P = C_k$, $Q = C_n$, this is the well-known partial order on the set of CSMs or ASMs. We will call $\quilt(P,Q)$ the \emph{quilt lattice}, as justified by the following.

\begin{theorem} \label{thm:lattice}
Let $P,Q$ be finite ranked posets with least and greatest elements.
The poset $\quilt(P,Q)$ is a distributive lattice ranked by $$\rk f =
\sum_{x \in P,\, y \in Q} f(x,y) - \sum_{x \in P,\, y \in Q} f_{\hat
0}(x,y),$$ where $f_{\hat 0}(x,y) = \max\{0,\rank x + \rank y -
\max\{\rank P,\rank Q\}\}$ is the least element of $\quilt(P,Q)$. The greatest
element of $\quilt(P,Q)$ is $f_{\hat 1}(x,y) = \min\{\rank x,\ \rank
y\}$.
\end{theorem}

Before we prove the theorem, let us introduce some notation. Given
$f,g \in \quilt(P,Q)$ such that $f \leq g$, define the \emph{difference set}
\begin{equation}\label{eq:difference.set}
\Delta(f,g) = \{(x,y) \in P \times Q \colon f(x,y) < g(x,y) \}.
\end{equation}
Given $(x,y) \in \Delta(f,g)$ and $(x',y') \in P \times Q$, write
$(x,y) \to (x',y')$ if either $(x,y) \lessdot (x',y')$ and $f(x,y) =
f(x',y')$ or $(x',y') \lessdot (x,y)$ and $f(x',y') = f(x,y) - 1$. In
the first case, $f(x',y') = f(x,y) < g(x,y) \leq g(x',y')$, which
means that $(x',y') \in \Delta(f,g)$. In the second case, $f(x',y') =
f(x,y) - 1 < g(x,y) - 1 \leq g(x',y')$, which also implies $(x',y')
\in \Delta(f,g)$. In other words, we have constructed a directed graph
$G_\Delta(f,g)$ on the difference set $\Delta(f,g)$. An edge from $(x,y)$ to
$(x',y')$ means that the value of $f$ on $(x',y')$ prevents us from
increasing the value of $f$ in $(x,y)$ to get another valid quilt of
type $(P,Q)$.

\begin{proof}[Proof of Theorem~\ref{thm:lattice}] To prove that $\quilt(P,Q)$ is a lattice, observe that for every $f,g \in
\quilt(P,Q)$, the pointwise minimum and maximum of $f$ and $g$ is
again in $\quilt(P,Q)$ by definition.  These quilts are $f\wedge g$ and $f \vee g$
respectively, hence $\quilt(P,Q)$ is a lattice. Distributivity follows
from the fact that $$\max\{a,\min\{b,c\}\} =
\min\{\max\{a,b\},\max\{a,c\}\}$$ and $$\min\{a,\max\{b,c\}\} =
\max\{\min\{a,b\},\min\{a,c\}\}$$ hold for $a,b,c \in \N$.

It is easy to check that $f_{\hat 0}$ and $f_{\hat 1}$ are in
$\quilt(P,Q)$ by checking the properties of the definition.  To see
$f_{\hat 1}$ is the unique maximal element in $\quilt(P,Q)$, consider
any $(x,y) \in P \times Q$ with $\rank x = r$, $\rank y = s$. Since
$P,Q$ are ranked posets, there exist maximal chains
\[
\hat 0_P = x_0 \lessdot x_1 \lessdot \dots
\lessdot x_r = x \lessdot \dots \lessdot x_k = \hat 1_P
\]
and
\[
\hat 0_Q = y_0 \lessdot y_1 \lessdot \dots \lessdot y_s = y \lessdot
\dots \lessdot y_n = \hat 1_Q.
\]
Then, for any $f \in \quilt(P,Q)$, we know
$f(x,y) \leq f(x_{r-1},y) + 1 \leq \ldots \leq f(\hat 0_P,y) + r = r$
and $f(x,y) \leq f(x,y_{s-1}) + 1 \leq \ldots \leq f(x,\hat 0_Q) + s =
s$, so $f(x,y) \leq \min\{r,s\} = f_{\hat 1}(x,y)$. Hence, $f_{\hat 1}$ is maximal.
On the other hand, $f(x,y) \geq f(x,y_{s+1}) - 1 \geq \dots \geq
f(x,\hat 1_Q) - (\rank Q-s) \geq f(x_{r+1},\hat 1_Q) - 1 - (\rank Q-s) \geq \dots
\geq f(\hat 1_P,\hat 1_Q) - (\rank P-r) - (\rank Q-s) = \min\{\rank P,\rank Q\} - \rank P - \rank Q + r + s = r + s - \max\{\rank P,\rank Q\}$ and so $f(x,y) \geq f_{\hat 0}(x,y)$.
Therefore, $f_{\hat 0}$ is minimal in $\quilt(P,Q)$.

To see $\quilt(P,Q)$ is ranked by the given function, choose $f,g \in
\quilt(P,Q)$ with $f < g$. Our goal is to prove that $f \lessdot g$ if and
only if $f$ and $g$ differ in exactly one $(x,y) \in P \times  Q$, and $g(x,y) =
f(x,y) + 1$. The condition is clearly sufficient, let us prove that it is also necessary.

Since $f<g$, the directed graph $G_\Delta(f,g)$ constructed above is
non-empty. We claim that it has no directed cycles, and that it
therefore has at least one sink.  To observe the claim, note that if
$(x,y) \to (x',y')$, then $f(x,y) \geq f(x',y')$, and if $f(x,y) =
f(x',y')$, then $\rank(x, y) < \rank(x', y')$ in $P \times Q$. If $(x,y)
\to (x',y') \to (x'',y'') \to \dots \to (x,y)$, then the value of $f$
must stay constant (if it strictly decreases, it can never increase to
$f(x,y)$ again), and that implies that $\rank(x,y) < \rank(x',y') <
\rank(x'',y'') < \dots < \rank(x,y)$, which is a contradiction.

Take $(x,y)$ to be an arbitrary sink in $G_\Delta(f,g)$.  Observe by
choice of $(x,y)$ that if $(x,y) \lessdot (x',y')$, then $f(x',y') =
f(x,y) + 1$, and if $(x',y') \lessdot (x,y)$, then $f(x',y') =
f(x,y)$. Since $f(\hat 0_{P},\hat 0_{Q})=g(\hat 0_{P},\hat 0_{Q})=0$
and $f(\hat 1_{P},\hat 1_{Q})=g(\hat 1_{P},\hat 1_{Q})=\min\{\rank P,\
\rank Q\}$ by definition of a quilt, $(x,y)$ is not $(\hat 0_{P},\hat
0_{Q})$ or $(\hat 1_{P},\hat 1_{Q})$.  Therefore, the function $f'
\colon P \times Q \longto \N$, defined to be equal to $f$ everywhere
except $f'(x,y)=f(x,y)+1$, is also a quilt of type $(P,Q)$, hence
$f\lessdot f'$.  Furthermore, since $(x,y)$ is a vertex in
$G_\Delta(f,g)$, we know $f'(x,y)=f(x,y)+1\leq g(x,y)$, so $f'\leq g$
as well. Therefore, we conclude that $f \lessdot g$ if and only if
$f'=g$.
\end{proof}

From the proof of the theorem, we learned that the sinks in the
difference graph $G_\Delta(f,g)$ determine all of the quilts covering $f$ in the
interval $[f,g]$.  This includes the case $g=f_{\hat 1}$ so all
covering relations in the quilt lattice are relatively easy to
identify.

\begin{cor}
 Given $f, g \in \quilt(P,Q)$ such that $f < g$, the atoms of the
 interval $[f,g]$ are in a one-to-one correspondence with the sinks of
 the graph $G_\Delta(f,g)$.   Similarly, the co-atoms of $[f,g]$ are
 in a one-to-one correspondence with the sources of $G_\Delta(f,g)$.
\qed
\end{cor}

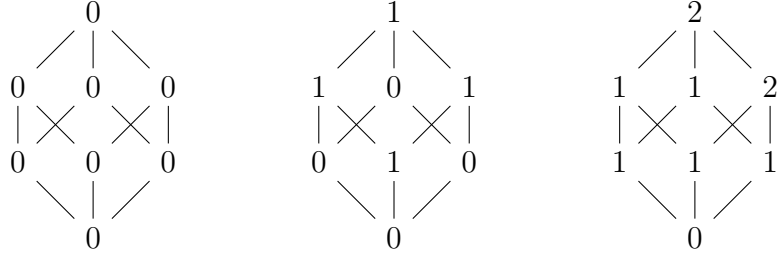
\begin{figure}
\begin{center}
\begin{tikzpicture}
\node (X0) at (0,0) {$0$};
\draw (-0.25,0.25) -- (-0.75,0.75);\draw (0,0.25) -- (0,0.75);\draw (0.25,0.25) -- (0.75,0.75);
\node (X11) at (-1,1) {$0$}; \node (X12) at (0,1) {$0$}; \node (X13) at (1,1) {$0$};
\draw (-1,1.25) -- (-1,1.75);\draw (-0.75,1.25) -- (-0.25,1.75);\draw (-0.25,1.25) -- (-0.75,1.75);\draw (0.25,1.25) -- (0.75,1.75);\draw (0.75,1.25) -- (0.25,1.75);\draw (1,1.25) -- (1,1.75);
\node (X21) at (-1,2) {$0$}; \node (X12) at (0,2) {$0$}; \node (X13) at (1,2) {$0$};
\draw (-0.25,2.75) -- (-0.75,2.25);\draw (0,2.75) -- (0,2.25);\draw (0.25,2.75) -- (0.75,2.25);
\node (X3) at (0,3) {$0$};
\node (X0) at (4,0) {$0$};
\draw (4-0.25,0.25) -- (4-0.75,0.75);\draw (4,0.25) -- (4,0.75);\draw (4.25,0.25) -- (4.75,0.75);
\node (X11) at (4-1,1) {$0$}; \node (X12) at (4,1) {$1$}; \node (X13) at (5,1) {$0$};
\draw (4-1,1.25) -- (4-1,1.75);\draw (4-0.75,1.25) -- (4-0.25,1.75);\draw (4-0.25,1.25) -- (4-0.75,1.75);\draw (4.25,1.25) -- (4.75,1.75);\draw (4.75,1.25) -- (4.25,1.75);\draw (5,1.25) -- (5,1.75);
\node (X21) at (4-1,2) {$1$}; \node (X12) at (4,2) {$0$}; \node (X13) at (5,2) {$1$};
\draw (4-0.25,2.75) -- (4-0.75,2.25);\draw (4,2.75) -- (4,2.25);\draw (4.25,2.75) -- (4.75,2.25);
\node (X3) at (4,3) {$1$};
\node (X0) at (8,0) {$0$};
\draw (8-0.25,0.25) -- (8-0.75,0.75);\draw (8,0.25) -- (8,0.75);\draw (8.25,0.25) -- (8.75,0.75);
\node (X11) at (8-1,1) {$1$}; \node (X12) at (8,1) {$1$}; \node (X13) at (9,1) {$1$};
\draw (8-1,1.25) -- (8-1,1.75);\draw (8-0.75,1.25) -- (8-0.25,1.75);\draw (8-0.25,1.25) -- (8-0.75,1.75);\draw (8.25,1.25) -- (8.75,1.75);\draw (8.75,1.25) -- (8.25,1.75);\draw (9,1.25) -- (9,1.75);
\node (X21) at (8-1,2) {$1$}; \node (X12) at (8,2) {$1$}; \node (X13) at (9,2) {$2$};
\draw (8-0.25,2.75) -- (8-0.75,2.25);\draw (8,2.75) -- (8,2.25);\draw (8.25,2.75) -- (8.75,2.25);
\node (X3) at (8,3) {$2$};
\end{tikzpicture}
\end{center}
\caption{A visual representation of an element in $\quilt(B_3,C_2)$}
\label{fig:quilt.B3.C2}
\end{figure}

We will call a quilt in $\quilt(P,C_n)$ or $\quilt(C_n,P)$ a
\emph{chain quilt}, a quilt in $\quilt(P,A_2(j))$ or
$\quilt(A_2(j),P)$ an \emph{antichain quilt}, a quilt in
$\quilt(P,B_n)$ or $\quilt(B_n,P)$ a \emph{Boolean quilt} etc.  As we
will explain in Section \ref{sec:motivation}, our most important and
motivating example will be when one of the posets is the Boolean
lattice and the other one is a chain.

There are three important ways to think of a chain quilt $f \in
\quilt(P,C_n)$. One is to see it as a sequence of Dedekind maps in
$D(P)$ that correspond with a walk in the Dedekind graph $G_D(P)$,
generalizing \Cref{prop:Dedekind.walks}.

\begin{example}\label{ex:B3.C2.quilt.and.ASMs}
\Cref{fig:quilt.B3.C2} shows a visual representation of the quilt $f \in
\quilt(B_{3},C_{2})$ given by
$$f(\{2 \},1)= f(\{1,2\},1) = f(\{2,3\},1) = f(\{1,2,3\},1) = 1$$ and
$f(T,1)=0$ for all other subsets $T$, while $f(\{2,3 \},2)=f(\{1,2,3
\},2)=2$ and $f(T,2)=1$ for all other nonempty subsets $T$.   From
\Cref{rem:ASMquiltname}, we know that for every maximal chain in $B_3$
(and the unique maximal chain in $C_2$), there is a corresponding $3
\times 2$ ASM encoded by the quilt $f$.  Here are the six ASMs paired
with maximal chains encoded by $f$,
\medskip

\begin{center}
$\left[\begin{smallmatrix} 0 & 1 \\ 1 & -1 \\ 0 &
1 \end{smallmatrix}\right]$ for $\begin{tikzpicture} \draw (0,0) --
(-0.08,0.08) -- (-0.08,0.16) -- (0,0.24); \end{tikzpicture}$,
$\left[\begin{smallmatrix}  1 & 0 \\ 0 & 0 \\ 0 &
1 \end{smallmatrix}\right]$ for $\begin{tikzpicture} \draw (0,0) --
(-0.08,0.08) -- (0,0.16) -- (0,0.24); \end{tikzpicture}$,
$\left[\begin{smallmatrix} 0 & 1 \\ 0 & 0 \\ 1 &
0 \end{smallmatrix}\right]$ for $\begin{tikzpicture} \draw (0,0) --
(0,0.08) -- (-0.08,0.16) -- (0,0.24); \end{tikzpicture}$,
$\left[\begin{smallmatrix} 0 & 0 \\ 0 & 1 \\ 1 &
0 \end{smallmatrix}\right]$ for $\begin{tikzpicture} \draw (0,0) --
(0,0.08) -- (0.08,0.16) -- (0,0.24); \end{tikzpicture}$,
$\left[\begin{smallmatrix} 1 & 0 \\ 0 & 0 \\ 0 &
1 \end{smallmatrix}\right]$ for $\begin{tikzpicture} \draw (0,0) --
(0.08,0.08) -- (0,0.16) -- (0,0.24); \end{tikzpicture}$, and
$\left[\begin{smallmatrix} 0 & 0 \\ 1 & 0 \\ 0 &
1 \end{smallmatrix}\right]$ for $\begin{tikzpicture} \draw (0,0) --
(0.08,0.08) -- (0.08,0.16) -- (0,0.24); \end{tikzpicture}.$
\end{center}
\end{example}

The second way to represent a chain quilt is to observe that $f$ maps
an arbitrary $x \in P$ to the sequence $(f(x,0),f(x,1),\ldots,f(x,n))$
of length $n+1$. This sequence has the property that every two
consecutive elements are either equal or they differ by one. We also
have $f(y,i) \in \{f(x,i),f(x,i)+1\}$ when $x \lessdot y$. The element
$\hat 0_P$ is mapped to the zero sequence, and the sequence
corresponding to $\hat 1_P$ ends with $\min\{\rank P,\ n\}$.  The
pictures in \Cref{fig:quilts.2.5} represent chain quilts with $P =
B_3$ for $n = 2$ and $n = 5$, respectively mapping $x \in B_{3}$ to
the sequence $(f(x,0),f(x,1),\ldots,f(x,n))$.  Note how the top
element on the left is $01\ldots n$, and the rightmost element of
every sequence on the right is equal to its rank, as stated in
\Cref{lem:chains.bounds}.  The picture on the right in
\Cref{fig:quilts.2.5} represents the same quilt at shown in
\Cref{fig:quilt.B3.C2}.

\bigskip

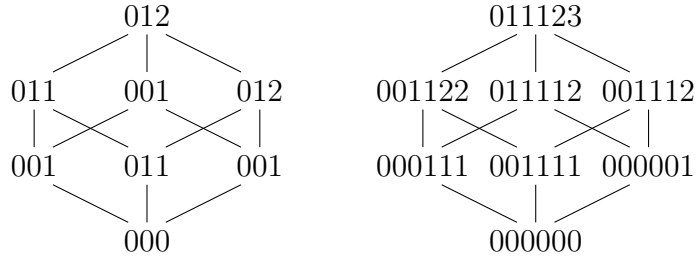
\begin{figure}
\begin{center}
\begin{tikzpicture}
\node (X0) at (0,0) {$000$};
\draw (-0.25,0.25) -- (-1.25,0.75);\draw (0,0.25) -- (0,0.75);\draw (0.25,0.25) -- (1.25,0.75);
\node (X11) at (-1.5,1) {$001$}; \node (X12) at (0,1) {$011$}; \node (X13) at (1.5,1) {$001$};
\draw (-1.5,1.25) -- (-1.5,1.75);\draw (-1.25,1.25) -- (-0.25,1.75);\draw (-0.25,1.25) -- (-1.25,1.75);\draw (0.25,1.25) -- (1.25,1.75);\draw (1.25,1.25) -- (0.25,1.75);\draw (1.5,1.25) -- (1.5,1.75);
\node (X21) at (-1.5,2) {$011$}; \node (X12) at (0,2) {$001$}; \node (X13) at (1.5,2) {$012$};
\draw (-0.25,2.75) -- (-1.25,2.25);\draw (0,2.75) -- (0,2.25);\draw (0.25,2.75) -- (1.25,2.25);
\node (X3) at (0,3) {$012$};
\end{tikzpicture} \qquad
\begin{tikzpicture}
\node (X0) at (0,0) {$000000$};
\draw (-0.25,0.25) -- (-1.25,0.75);\draw (0,0.25) -- (0,0.75);\draw (0.25,0.25) -- (1.25,0.75);
\node (X11) at (-1.5,1) {$000111$}; \node (X12) at (0,1) {$001111$}; \node (X13) at (1.5,1) {$000001$};
\draw (-1.5,1.25) -- (-1.5,1.75);\draw (-1.25,1.25) -- (-0.25,1.75);\draw (-0.25,1.25) -- (-1.25,1.75);\draw (0.25,1.25) -- (1.25,1.75);\draw (1.25,1.25) -- (0.25,1.75);\draw (1.5,1.25) -- (1.5,1.75);
\node (X21) at (-1.5,2) {$001122$}; \node (X12) at (0,2) {$011112$}; \node (X13) at (1.5,2) {$001112$};
\draw (-0.25,2.75) -- (-1.25,2.25);\draw (0,2.75) -- (0,2.25);\draw (0.25,2.75) -- (1.25,2.25);
\node (X3) at (0,3) {$011123$};
\end{tikzpicture}
\end{center}
\caption{Another visualization of quilts of types $(B_{3},C_{2})$ and
$(B_{3},C_{5})$.}
\label{fig:quilts.2.5}
\end{figure}

Another equivalent, and probably even more intuitive, way to represent a chain quilt $f:P \times
C_{n} \longrightarrow \mathbb{N}$ is to say that it is a map that
sends $x \in P$ to the set of \emph{jumps of $f$ at $x$},
\begin{equation}\label{eq:jumps.at.x}
J_f(x) = \{ i \in [n]\colon f(x,i) = f(x,i-1) + 1\}.
\end{equation}
We will call this
the \emph{monotone triangle (MT) form} of the quilt $f$. It is easy
to go back from the MT form of a quilt: given $J:P \longrightarrow B_{n}$,
then $f(x,i)=|J(x)\cap [i]|$ defines $f \colon P \times C_n \longrightarrow \mathbb{N}$.
The Boolean growth condition translates into adjacent sets interlacing for quilts in MT form, see \Cref{prop:interlacing}.

The following shows the two quilts from \Cref{fig:quilts.2.5} in MT form, where we
omit braces and commas for sets since all integers are below $10$.

\begin{center}
\begin{tikzpicture}
\node (X0) at (0,0) {$\emptyset$};
\draw (-0.25,0.25) -- (-1.25,0.75);\draw (0,0.25) -- (0,0.75);\draw (0.25,0.25) -- (1.25,0.75);
\node (X11) at (-1.5,1) {$2$}; \node (X12) at (0,1) {$1$}; \node (X13) at (1.5,1) {$2$};
\draw (-1.5,1.25) -- (-1.5,1.75);\draw (-1.25,1.25) -- (-0.25,1.75);\draw (-0.25,1.25) -- (-1.25,1.75);\draw (0.25,1.25) -- (1.25,1.75);\draw (1.25,1.25) -- (0.25,1.75);\draw (1.5,1.25) -- (1.5,1.75);
\node (X21) at (-1.5,2) {$1$}; \node (X12) at (0,2) {$2$}; \node (X13) at (1.5,2) {$12$};
\draw (-0.25,2.75) -- (-1.25,2.25);\draw (0,2.75) -- (0,2.25);\draw (0.25,2.75) -- (1.25,2.25);
\node (X3) at (0,3) {$12$};
\end{tikzpicture} \qquad
\begin{tikzpicture}
\node (X0) at (0,0) {$\emptyset$};
\draw (-0.25,0.25) -- (-1.25,0.75);\draw (0,0.25) -- (0,0.75);\draw (0.25,0.25) -- (1.25,0.75);
\node (X11) at (-1.5,1) {$3$}; \node (X12) at (0,1) {$2$}; \node (X13) at (1.5,1) {$5$};
\draw (-1.5,1.25) -- (-1.5,1.75);\draw (-1.25,1.25) -- (-0.25,1.75);\draw (-0.25,1.25) -- (-1.25,1.75);\draw (0.25,1.25) -- (1.25,1.75);\draw (1.25,1.25) -- (0.25,1.75);\draw (1.5,1.25) -- (1.5,1.75);
\node (X21) at (-1.5,2) {$24$}; \node (X12) at (0,2) {$15$}; \node (X13) at (1.5,2) {$25$};
\draw (-0.25,2.75) -- (-1.25,2.25);\draw (0,2.75) -- (0,2.25);\draw (0.25,2.75) -- (1.25,2.25);
\node (X3) at (0,3) {$145$};
\end{tikzpicture}
\end{center}

\begin{prop} \label{prop:interlacing}
 Take $f \in \quilt(P,C_n)$. For all $x,y \in P$ with $x \lessdot y$, the sets $S = J_f(x)$ and $T = J_f(y)$ are interlacing. When $n \leq \rank P$, $ J_f(\hat 1_P) = [n]$. When $n \geq \rank P$, we have $|J_{f}(x)| = \rank x$ for all $x \in P$.
\end{prop}

\begin{proof}
Assume $f \in \quilt(P,C_n)$.  By construction, we have $|S| = f(x,n)$
and $|T| = f(y,n)$.  Since $x \lessdot y$ and $f$ is a quilt, $|T| =
|S|$ or $|T| = |S|+1$. Furthermore, the $i$-th jump in $J_{f}(y)$
cannot come after the $i$-th jump in $J_{f}(x)$. That implies that
$t_i \leq s_i$. Also, the $i$-th jump in $J_{f}(x)$ cannot come after
the $(i+1)$-st jump in $J_{f}(y)$, as that would violate the rule that
$f(x,j)$ and $f(y,j)$ can differ by at most $1$. Therefore $s_i \leq
t_{i+1}$. The last two statements follow from
\Cref{lem:chains.bounds}.
\end{proof}

The general problem of enumerating quilts is hard, as the following
theorem shows. Recall that a counting problem is in \#P if we can
represent it as counting the number of accepting paths of a
polynomial-time non-deterministic Turing machine, and it is
\#P-complete if every problem in \#P has a polynomial-time counting
reduction to it.  To prove a problem is $\#P$-complete, it suffices to
show it is in $\#P$ and is as hard as some $\#P$-complete problem.

\begin{theorem} \label{thm:complexity}
 Computing $|\quilt(P,Q)|$ for general $P$ and $Q$ is a \#P-complete problem.
\end{theorem}
\begin{proof}
One can check if a given map $P \times Q \to \N$ satisfies the
properties of a quilt in polynomial time in terms of the sizes of $P$
and $Q$. It follows that the problem
of computing $|\quilt(P,Q)|$ is in \#P. To prove \#P-completeness,
note that mapping $f \in \quilt(P, C_1)$ to the set of minimal
elements of $\{ x \in P \colon f(x,\hat 1) = 1\}$ gives a bijection
between $\quilt(P, C_1)$ and the set $\{A \subseteq P, A \mbox{
antichain}\} \setminus \{\emptyset, \{\hat 0_P\}\}$. The fact that
counting antichains in finite posets is \#P-complete is Part 3 of the
main theorem proved by Provan and Ball in \cite[p. 783]{antichains}. One
can easily adapt their proof to ranked posets with $\hat 0$ and $\hat
1$: the poset constructed in the proof is already ranked, and adding
$\hat 0$ and $\hat 1$ just adds two extra antichains.
\end{proof}

Even though computing $|\quilt(P,Q)|$ for general $P$ and $Q$ is
out of reach, we do have a simple upper bound. Recall that $b(P) = \sum_{x \in P}
\rank x$.

\begin{theorem} \label{thm:upperbound}
 If $\rank P \leq \rank Q$, then $|\quilt(P,Q)| \leq d_1(Q)^{b(P)}$.
\end{theorem}
\begin{proof}
Consider a quilt $f \in \quilt(P,Q)$. For $x \in P$, we have $f(x,\hat 1_Q) = \rank x$, and the map $f^x \colon y \mapsto f(x,y)$ is in $D_{\rank x}(Q)$. The map $f \mapsto (f^x)_{x \in P}$ is an injection, so $|\quilt(P,Q)| \leq \prod_{x \in P} d_{\rank x}(Q)$. By \Cref{lemma:dedekindupper},
$$|\quilt(P,Q)| \leq \prod_{x \in P} d_{\rank x}(Q) \leq \prod_{x \in P} d_1(Q)^{\rank x} = d_1(Q)^{b(P)}.$$
\end{proof}

\section{Motivation from the space of spanning line configurations}
\label{sec:motivation}

An interesting analog of the flag manifold and Schubert varieties was
given by Pawlowski and Rhoades in \cite{PR}.  For any $k\leq n$, they
define a \emph{spanning line configuration} $l_{\bullet} =
(l_1,\ldots,l_n)$ to be an ordered $n$-tuple in the product of complex
projective spaces $(\mathbb{P}^{k-1})^{n}$ whose vector space sum is
$\mathbb{C}^k$.  Each such configuration can be identified by a $k
\times n$ full rank matrix over $\mathbb{C}$ with no zero columns.
Note, multiplying a $k \times n$ matrix on the right by an invertible
$n \times n$ diagonal matrix determines the same spanning line
configuration.  Therefore, the space of all such configurations can be
identified with the orbits
\begin{equation}X_{n,k}= \Fullrk/T = \{l_{\bullet} = (l_1,\ldots,l_n)
\in (\mathbb{P}^{k-1})^n \hspace{0.05in} \colon \hspace{0.05in} l_1 +
\cdots + l_n = \mathbb{C}^k\}
\end{equation}
where $\Fullrk$ is the set of full rank $k \times n$ matrices with no
zero columns and $T$ is the set of diagonal matrices in
$GL_{n}(\mathbb{C})$.

Pawlowski and Rhoades give a cell decomposition of $X_{n,k}=\bigcup
C_{w}$ indexed by Fubini words $w \in W_{n,k}$, which were defined in
\Cref{sub:asm.fubini}.  This cell decomposition plays an important
role in the geometry and topology of the space of spanning line
configurations.  The closure of the cell $C_{w}$ in $X_{n,k}$ is
determined by certain rank conditions on the matrices representing
points in $C_{w}$.  The Pawlowski-Rhoades varieties, or \emph{PR
varieties} for short, are the cell closures $\overline{C}_{w}$.

The required rank conditions for PR varieties can be described in
terms of quilts as follows.  Every $k \times n$ matrix $M$ determines
a Boolean-chain quilt $f_{M}: C_{k} \times B_{n} \longrightarrow
\mathbb{N}$ given by sending $(h,J)$ to the rank of the submatrix of M
on rows $[h]=\{1,2,\dotsc ,h \}$ and columns in $J$, denoted
$f_{M}(h,J)=\rank(M[[h],J])$.  Here we define the boundary cases by
$f_{M}(0,J)=f_{M}(h,\emptyset) =0$ for all $1\leq h\leq k$ and $J
\subseteq [n]$.  The quilt $f_{M}$ also encodes the rank functions of
the flag matroid corresponding to $M$ as mentioned in
\Cref{sub:matroids}.  Furthermore, if $D \in T$ is an invertible
diagonal matrix then $f_{M}=f_{MD}$, so such collections of rank
functions are constant on the $T$-orbits in $\Fullrk$. Hence, each
spanning line configuration gives rise to a well-defined Boolean-chain
quilt.  It was shown in \cite[Lemma 4.1.2]{Ryan.2022} that $f_{M}$
determines which cell $C_{w} \subset X_{n,k}$ contains the spanning
line configuration determined by the columns of $M$ for each $M \in
\Fullrk$.  On the other hand, if we define a function $C_{k}\times
B_{n} \longrightarrow \mathbb{N}$ for each Fubini word $w \in W_{n,k}$
by
\begin{equation}\label{eq:def.fw}
f_{w}(h,J) = \mathrm{max}\{f_{M}(h,J) \given M \in C_{w} \} 
\end{equation}
for each $1\leq h\leq k$ and $J \subseteq [n]$, then we have the
following observations.

\begin{lemma}\label{lem:medium.roats.quilt.lattice}
For $w \in W_{n,k}$, the map $f_{w} \in \quilt(C_{k},B_{n})$.
Furthermore, the spanning line configuration defined by $M \in \Fullrk$ is
in $\overline{C}_{w}$ if and only if $f_{M} \leq f_{w}$ in the quilt
lattice of type $(C_{k},B_{n})$.
\end{lemma}

\begin{proof}
The fact that $f_{w}$ is a quilt of type $(C_{k},B_{n})$ follows from
the fact that $f_{w}=f_{M}$ for any generically chosen $M \in C_{w}$.
By \eqref{eq:def.fw}, if $M \in C_{w}$, then $f_{M}\leq f_{w}$ in the
quilt lattice. This claim extends any $M \in \overline{C}_{w}$ since
the closure is defined by such rank conditions.  Conversely, if $M \in
\Fullrk$ and $f_{M} \leq f_{w}$ in the quilt lattice, then $M$
satisfies the defining rank conditions for the PR variety $\overline{C}_{w}$.
\end{proof}

There is a natural analog of Bruhat order for Fubini words in
$W_{n,k}$ given by the (reverse) containment order on the PR
varieties, so $v\leq w$ if and only if $\overline{C}_{w} \subseteq
\overline{C}_{v}$.  This partial order, originally studied by
Pawlowski--Rhoades, is called the \textit{medium roast Fubini--Bruhat
order} on $W_{n,k}$ following terminology in
\cite{Billey-Ryan,Ryan.2022}.  Billey--Ryan observed that the medium
roast order can also be determined in terms of certain vanishing flag
minors.   An algorithm to determine $f_{w}:C_{k} \times
B_{n}\longrightarrow C_{k}$ directly from $w$ is given by \cite[Lemma
5.2.8]{Ryan.2022}.

\begin{cor}\label{cor:medium.roast.quilts}
The medium roast order on $W_{n,k}$ is the subposet of
$\quilt(C_k,B_n)$ defined by
\[
v\leq w \iff f_{v} \geq f_{w}.
\]
\end{cor}

\begin{cor}\label{cor:intervals.in.medium}
The interval $[v,w]$ in medium roast Fubini order on $W_{n,k}$ can be
determined from the interval $[f_{v},f_{w}]$ in $\quilt(C_k,B_n)$ by
identifying all of the quilts in $[f_{v},f_{w}]$ which correspond to
some Fubini word.  In particular, testing if $v$ is covered by $w$
reduces to checking that the open interval $(f_{v},f_{w})$ contains no
$f_{u}$ for $u \in W_{n,k}$.
\end{cor}

\begin{remark}
Finding a concise description of all covering relations in the medium
roast order on $W_{n,k}$ is still an open problem as of the writing of
this paper.  Can the embedding into the quilt lattice
$\quilt(C_k,B_n)$ be used to find such a characterization?
\end{remark}


\section{Properties of the quilt lattice} \label{sec:lattice}

In this section, we prove some further properties of the lattice
$\quilt(P,Q)$ using natural involutions and embeddings.  We begin with
some examples of how the CSM/ASM poset, the medium roast poset, and a
natural poset on matroids all embed into quilt lattices.

Observe from the definitions that for all posets $P,Q$ with $\rank P =
k$ and $\rank Q = n$, the map
\begin{equation}\label{eq:iota}
\iota \colon \CSM_{k,n} \longto \quilt(P,Q), \qquad \iota(f)(x,y) = f(\rank x,\ \rank y)
\end{equation}
 is a lattice embedding.  This embedding is not necessarily
surjective; in fact, it can be quite sparse.  For example,
\Cref{fig:nk32} shows the lattice $\quilt(C_2,B_3)$, which has $199$
elements. For $k\leq n$, we illustrate that $\quilt(C_k,B_n)$ contains
(among others) the following three overlapping subposets.
\begin{itemize}
 \item ASMs/CSMs of size $k \times n$, embedded via $\iota$ from
\eqref{eq:iota} (the seven ASMs of size $2 \times 3$ are
marked with red and purple);
\item The Fubini words $W_{n,k}$ with the medium roast order on Fubini words are embedded
into $\quilt(C_k,B_n)$ by \Cref{lem:medium.roats.quilt.lattice}.  The
six Fubini words of length $3$ with letters $1$ and $2$ are marked
with blue and purple.
\item Matroids on ground set $[n]$ with rank $k$ are embedded into
$\quilt(C_k,B_n)$ via the map that sends a matroid on $[n]$ to the
quilt $f:C_{k}\times B_{n} \longrightarrow \mathbb{N}$ with $f(i,T) =
\min\{i,\ \rank T\}$, where $\rank T$ is the cardinality of the largest
independent set contained in $T$. The seven matroids on $[3]$ with
rank $2$ are marked as squares.

\end{itemize}

\begin{figure}[!ht]
\centering
\includegraphics[width=.95\linewidth]{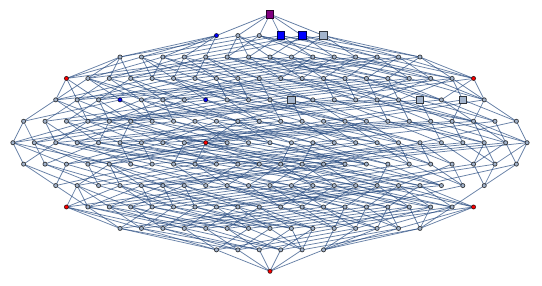}
\caption{The Hasse diagram of $\quilt(C_2,B_3)$.}
\label{fig:nk32}
\end{figure}

The rank functions of flag matroids $\mathcal{M}$ on the ground set
$[n]$ of type $(k_1< \cdots < k_s)$ and rank functions
$\rank_1,\ldots, \rank_s$ can also be encoded as a Boolean-chain
quilt. Specifically, the embedding into $\quilt(C_{k_{s}},B_n)$ is the
map that sends $\mathcal{M}$ to $f_{\mathcal{M}}$, where
\[
f_{\mathcal{M}}(i,X) = \begin{cases}
\min\{ i, \rank_1 X\} &  1\leq i\leq k_{1}  \\
\min\{ \rank_j X +i-k_{j} , \rank_{j+1} X\} & k_{j}< i\leq k_{j+1}, 1 \leq j \leq s-1.
\end{cases}
\]

One can verify $f_{\mathcal{M}}$ satisfies the properties of a quilt
from the definitions in \Cref{sub:matroids} as follows.
By definition,
$f_{\mathcal{M}}(0,X)=f_{\mathcal{M}}(i,\emptyset)=\rank_j(\emptyset)
=0$ for all $i,j,X$ and $f_{\mathcal{M}}(k_{s},[n])=
\rank_{s}[n]=k_{s}$.  Matroid rank functions satisfy Boolean growth,
hence $f_{\mathcal{M}}(i,X\cup \{y \})-f_{\mathcal{M}}(i,X) \in \{0,1
\}$ for all $y \in [n]\setminus X$.  Finally, observe that $\min\{
\rank_j X +i-k_{j},\ \rank_{j+1} X\}$ exhibits Boolean growth as a
function of $i$ since $\rank_j X\leq \rank_{j+1} X$ and because
$M_{j}$ being a quotient of $M_{j+1}$ implies $\rank_j X + k_{j+1} -
k_j \geq \rank_{j+1} X$ by \eqref{eq:flag.matroid.rank.inequality}.

Under certain circumstances, (anti)automorphisms of the posets $P$ and
$Q$ give us (anti)auto\-mor\-phisms of the quilt lattice. Here an
antiautomorphism of a ranked poset $P$ is an isomorphism between $(P,\leq)$
and its dual poset $(P,\geq)$.  If there exists an involutive antiautomorphism on $P$, there is a dihedral group
$D_4$ action on $\quilt(P,P)$, like with square ASMs.

\begin{theorem}
 Let $P$ and $Q$ be finite ranked posets with least and greatest
elements.
 \begin{enumerate}
 \item The switch map
 $$\Sigma \colon \quilt(P,Q) \to \quilt(Q,P), \quad \text{where }\Sigma(f)(x,y) = f(y,x)$$
 is an involutive lattice isomorphism.
 \item If $\gamma$ is an automorphism of $P$, then
 $$\Gamma \colon  \quilt(P,Q) \to \quilt(P,Q), \quad \text{where }\Gamma(f)(x,y) = f(\gamma(x), y)$$
 is an automorphism of the lattice $\quilt(P,Q)$.
 \item If $\varphi$ is an (involutive) antiautomorphism of $P$ and $\rank P \geq \rank Q$, then
 $$\Phi \colon \quilt(P,Q) \to \quilt(P,Q), \quad \text{where } \Phi f(x,y) = \rank y - f(\varphi(x), y)$$
 is an (involutive) antiautomorphism of the lattice $\quilt(P,Q)$.
 \item Given an involutive antiautomorphism $\varphi \colon P \to P$, there is an action of the dihedral group $D_4$ acting on $\quilt(P,P)$ that sends the horizontal reflection of the square to $\Phi$ and the diagonal reflection of the square to $\Sigma$. If $\rank P \geq 2$, the action is faithful.
 \end{enumerate}
\end{theorem}
\begin{proof}
 Parts (1) and (2) are straightforward to verify from
\Cref{def:quilt}. For (3) one can verify the conditions of
\Cref{def:quilt} as follows. Note that we have $\Phi f(x,\hat 0_Q) =
\rank \hat 0_Q - f(\varphi(x), \hat 0_Q) = 0$, $\Phi f(\hat 0_P,y) =
\rank y - f(\varphi(\hat 0_P),y) = \rank y - f(\hat 1_P,y) = \rank y -
\rank y = 0$ and $\Phi f(\hat 1_P, \hat 1_Q) = \rank \hat 1_Q -
f(\varphi(\hat 1_P),\hat 1_Q) = \rank Q - f(\hat 0_{P}, \hat 1_Q) =
\min \{\rank P,\ \rank Q\}$. If $x \lessdot x'$, then $\varphi(x')
\lessdot \varphi(x)$, so $\Phi f(x',y) - \Phi f(x, y) = f(\varphi(x),
y) - f(\varphi(x'), y) \in \{0,1\}$. If $y \lessdot y'$, then
$1=\rank y' - \rank y$, so
\[
\Phi f(x,y') - \Phi f(x,y) = 1 - (f(\varphi(x), y') - f(\varphi(x),
y)) \in \{0,1\}.
\]
That means that $\Phi f \in \quilt(P,Q)$. If $f \leq g$, then $\Phi
f(x,y) = \rank y - f(\varphi(x),y) \geq \rank y - g(\varphi(x),y) =
\Phi g(x,y)$. If $\lambda$ is the inverse of $\varphi$, then $\Lambda
\colon \quilt(P,Q) \to \quilt(P,Q)$, where $(\Lambda f)(x,y) = \rank y
- f(\lambda(x),y))$, is easily seen to be the inverse of $\Phi$. In
particular, if $\varphi$ is an involution, so is $\Phi$.

To prove (4), compute $$(\Sigma \circ \Phi) f(x,y) = \Phi f(y,x) =
\rank x - f(\varphi(y), x),$$

$$(\Sigma \circ \Phi)^2 f(x,y) = \rank x - (\Sigma \circ \Phi)
f(\varphi(y),x) = \rank x - \rank \varphi(y) +
f(\varphi(x),\varphi(y)),$$

$$(\Sigma \circ \Phi)^3 f(x,y) = \rank x - (\Sigma \circ \Phi)^2
f(\varphi(y),x) = \rank x - \rank \varphi(y) + \rank \varphi(x) -
f(y,\varphi(x)),$$

$$(\Sigma \circ \Phi)^4 f(x,y) = \rank x - (\Sigma \circ \Phi)^3
f(\varphi(y),x) = \rank x - \rank \varphi(y) + \rank \varphi(x) -
\rank y + f(x,y),$$

\noindent and because $\rank x + \rank \varphi(x) = \rank y
+ \rank \varphi(y) = \rank P$, we know $(\Sigma \circ \Phi)^4 =
\id$. Therefore $D_4$ acts on $\quilt(P,P)$.

Finally, $\quilt(P,P)$ contains a representative of all square CSMs of
size $\rank P \times \rank P$ under the map $\iota$ defined in \eqref{eq:iota}, and the symmetries of the square act
faithfully on this subset if $\rank P \geq 2$.  Therefore, they also act faithfully on
$\quilt(P,P)$.
\end{proof}

\begin{example}
 Any permutation of the labels gives an automorphism of $B_k$, which gives many automorphisms of $\quilt(B_k,Q)$ for any $Q$. The maps $i \mapsto k - i$ and $T \mapsto [k] \setminus T$ are involutive antiautomorphisms on $C_k$ and $B_k$, which gives involutive automorphisms of $\quilt(C_k,Q)$ and $\quilt(B_k,Q)$ for $\rank Q \leq k$ and a faithful $D_4$ action on $\quilt(B_n,B_n)$ for $n \geq 2$.
\end{example}

Take posets $P_1$ and $P_2$ with the same rank. The \emph{disjoint union} $P_1 + P_2$ is the poset we get my ``merging'' $\hat 0_{P_1}$ with $\hat 0_{P_2}$ and $\hat 1_{P_1}$ with $\hat 1_{P_2}$, and adding the other elements of $P_1$ and $P_2$ without any new relations. For example, $A_2(j_1) + A_2(j_2)$ is isomorphic to $A_2(j_1 + j_2)$. Write $j P$ for the disjoint union of $j$ copies of $P$. For example, $A_2(j) = j C_2$.

\begin{prop}
 Assume that $\rank P_1 = \rank P_2 \geq \rank Q$. Then the map
 $$\Theta \colon \quilt(P_1 + P_2, Q) \longto \quilt(P_1, Q) \times \quilt(P_2, Q)$$
 defined by
 $$f \mapsto (f_1,f_2), \qquad f_i(x_i,y) = f(x_i,y) \text{ for } x_i \in P_i, y \in Q,$$
 is an isomorphism of lattices.
\end{prop}
\begin{proof}
  The only non-trivial part to prove is that the map is
  surjective. Say that we are given $f_1 \in \quilt(P_1,Q)$ and $f_2
  \in \quilt(P_2,Q)$. Because $\rank P_i \geq \rank Q$, $f_i(\hat
  1_{P_i},y) = \rank y$ for all $y \in Q$. That means that $f_1$ and
  $f_2$ are compatible in the only two ``common'' elements of $P_1$
  and $P_2$ in $P_1+P_2$, and the map $f:(P_1 + P_2)\times  Q
  \longrightarrow \mathbb{N}$ given by
  $$f(x,y) = \left\{ \begin{array}{ccl} f_1(x,y) & : & x \in P_1 \\ f_2(x,y) & : & x \in P_2 \end{array} \right.$$
  is a well-defined quilt and maps to $(f_1,f_2)$.
\end{proof}

\begin{remark}\label{rem:sums}
Without the assumption that $\rank P_1 = \rank P_2 \geq \rank Q$, the
map $\Theta$ is still a well-defined injective homomorphisms of
lattices, but is not necessarily a surjection. Therefore, we have
$|\quilt(P_1 + P_2, Q)| \leq |\quilt(P_1, Q)| \cdot |\quilt(P_2, Q)|$.
\end{remark}

\begin{cor} \label{cor:powerofasm}
 For $k \geq n$ and arbitrary positive integer $j$, $|\quilt(j C_k,
C_n)| = |\ASM_{k \times n}|^{j}$. For any $n,i,j$, we have $|\quilt(i
C_n, j C_n)| = |\ASM_{n \times n}|^{ij}$.
\end{cor}

\begin{prop}
 Assume that a map $\psi \colon Q' \to Q$ has the following properties:
 \begin{itemize}
  \item $\psi$ is surjective,
  \item $x \lessdot y \Rightarrow \psi(x) = \psi(y) \mbox{ or } \psi(x) \lessdot \psi(y)$.
 \end{itemize}
 For $P$ with $\rank P \leq \rank Q$, the map
 $$\Psi \colon \quilt(P,Q) \longto \quilt(P,Q'), \qquad \text{where }
 \Psi f(x,y') = f(x,\psi(y')) \text{ for all } f \in \quilt(P,Q)$$
 is an injective lattice homomorphism.
\end{prop}
\begin{proof}
 Let us first check that $\Psi f$ is indeed a quilt. It follows from
the assumptions that $\psi(\hat 0_{Q'}) = \hat 0_Q$ and $\psi(\hat
1_{Q'}) = \hat 1_Q$, so $\rank Q \leq \rank Q'$.  We have $\Psi f(\hat
0_P,y) = f(\hat 0_P,\psi(y)) = 0$, $\Psi f(x, \hat 0_{Q'}) =
f(x,\psi(\hat 0_{Q'})) = f(x, \hat 0_Q) = 0$ and $$\Psi f(\hat 1_P,
\hat 1_{Q'}) = f(\hat 1_P,\hat 1_Q) = \min\{\rank P,\ \rank Q\} = \rank
P = \min\{\rank P,\ \rank Q'\}.
$$
If $x_1 \lessdot x_2$ in $P$, then $\Psi f(x_2,y') = f(x_2,\psi(y'))$
is either $f(x_1,\psi(y'))$ or $f(x_1,\psi(y')) + 1$. On the other
hand, if $y'_1 \lessdot y'_2$ in $Q'$, then $\psi(y'_1) = \psi(y'_2)$
or $\psi(y'_1) \lessdot \psi(y'_2)$. In both cases, either $\Psi
f(x,y'_2) = \Psi f(x,y'_1)$ or $\Psi f(x,y'_2) = \Psi f(x,y'_1) +
1$. That completes the verification that $\Psi f$ is a quilt.

 If $f \leq g$ in $\quilt(P,Q)$, then for $x \in P$, $y' \in Q'$,
$\Psi f(x,y') = f(x,\psi(y')) \leq g(x,\psi(y')) = \Psi g(x,y')$, so
$\Psi f \leq \Psi g$. If $\Psi f = \Psi g$ and $x \in P$, $y \in Q$
are arbitrary, then $y = \psi(y')$ for some $y' \in Q'$ by the
surjectivity of $\psi$, and $f(x,y) = \Psi f(x,y') = \Psi g(x,y') =
g(x,y)$. That means that $\Psi$ is an injective homomorphism on two
quilt lattices.
\end{proof}

\newcommand{\mm}{n}
\newcommand{\mn}{m}

\begin{example}
 Take $\mn \leq \mm$. The following maps have the required
properties: \begin{itemize} \item $\psi_1 \colon C_\mm \to C_\mn$,
$\psi(i) = \min\{i,\mn\}$, \item $\psi_2 \colon B_\mm \to B_\mn$,
$\psi(T) = T \cap [\mn]$ \end{itemize} We therefore have injective
lattice homomorphisms $$\Psi_1 \colon \quilt(P, C_\mn) \longto
\quilt(P, C_\mm) \mbox{ and } \Psi_2 \colon \quilt(P, B_\mn) \longto
\quilt(P, B_\mm).$$ for $\rank P \leq \mn$. Since $\psi_1$ has the
added property that $\psi_1(i) = \psi_1(j) \Rightarrow i = j \mbox{ or
} \psi_1(i) = \psi_1(j) = \mn$, $\Psi_1$ preserves cover relations.

Thus, the embedding of $\quilt(P, C_\mn)$ into $\quilt(P, C_\mm)$ is
\textit{isometric}, meaning $\rank g - \rank f = \rank \Psi g - \rank
\Psi f$ for all $f,g$. For $P = C_k$, $k \leq \mn \leq \mm$, this is equivalent to
 adding $\mm-\mn$ zero columns to a $k \times \mn$ ASM to get a $k
 \times \mm$ ASM. The map $\Psi_2$ does not preserve covering relations: take quilts $f,g \in \quilt(C_2,B_2)$ satisfying $f(1,\{1\}) = 0$, $g(1,\{1\}) = 1$ and $f \lessdot g$; then $\Psi_2 f(1,\{1\}) = \Psi_2 f(1,\{1,3\}) =0$, $\Psi_2 g(1,\{1\}) = \Psi_2 g(1,\{1,3\}) = 1$, so $\Psi_2 f \not\!\!\lessdot \Psi_2 g$.
\end{example}

\section{Enumeration of antichain quilts} \label{sec:antichain}

In this section, we consider the case of counting the number of quilts
of type $(P,Q)$ when $Q$ is an antichain poset.  The enumeration is in
terms of the number of antichains in convex cut sets of $P$.  We begin
by defining the necessary vocabulary and notation.  Several specific
examples are included following the corollaries.

We say that a subset $S$ of a poset $P$ is \emph{convex} if $x,y \in S$
implies $[x,y] \subseteq S$. We say that $S$ is a \emph{cut set} if it
intersects every maximal chain in $P$. If
you have a convex cut set $C$, it makes sense to say that an element
$x \in P \setminus C$ is \emph{above $C$} or \emph{below $C$}: $x$
lies on a maximal chain, the maximal chain intersects $C$ in some
element $x'$, and $x$ is above $C$ if $x > x'$ and below $C$ if $x <
x'$. This is well defined, as $x' < x < x''$ for $x',x'' \in C$ would
imply $x \in C$. For example, if $\rank P \geq 2$, then $C=P \setminus
\{\hat 0_P, \hat 1_P\}$ is a convex cut set, $\hat 0_P$ is below $C$,
and $\hat 1_P$ is above $C$.

Recall from \Cref{sec:def} that $d_1(P)$ counts the number of nonempty
antichains in $P$ other than $\{\hat 0\}$. Such antichains are in
bijection with antichains in $P \setminus \{\hat 0_P, \hat 1_P\}$.
For any $S \subseteq P$, denote by $\alpha_P(S)$ the number of
antichains in $S$.  Then, we have $\alpha_P(P \setminus \{\hat 0_P,
\hat 1_P\}) = d_1(P)$.  Given two infinite sequences $(a_{n})$ and
$(b_{n})$, we write $a_{n} \sim b_{n}$ to mean $a_{n}/b_{n}
\rightarrow 1$ as $n$ goes to infinity.

\begin{theorem} \label{thm:antichain}
 Take a ranked poset $P$ with least and greatest elements, $\rank P \geq 2$, and $j \geq 1$. We have
 \begin{equation} \label{eq:antichain}
 |\quilt(P,A_2(j))| = \sum_C \alpha_P(C)^j,
 \end{equation}
 where the sum is over all subsets $C$ of $P \setminus \{\hat 0_P,\hat
 1_P\}$ that are convex cut sets of $P$. In particular, as $j$ goes to
 infinity, we have
\begin{equation} \label{eq:antichain2}
|\quilt(P,A_2(j))| \sim  d_1(P)^j.
\end{equation}
\end{theorem}

\begin{proof}
 For $f \in \quilt(P,A_2(j))$, we have $f(x,y) \leq \min\{\rank
P,\ \rank A_2(j)\} = 2$ for all $x \in P$ and $y \in A_2(j)$. Take $x
\in P$. If $f(x,\hat 1) = 0$, then $f(x,y) = 0$ for all $y \in
A_2(j)$. If $f(x,\hat 1) = 2$, then $f(x,\hat 0) = 0$ and $f(x,y) = 1$
for $\rank y = 1$. If, however, $f(x, \hat 1) = 1$, then $f(x,y)$ can
be either $0$ or $1$ for $\rank y = 1$.  If $x \leq x''$ and $f(x,
\hat 1) = f(x'', \hat 1) = 1$, then $f(x', \hat 1) = 1$ for every $x'
\in [x,x'']$ since $f \in \quilt(P,A_2(j))$.  Furthermore, for a
maximal chain $\hat 0_P = x_0 \lessdot x_1 \lessdot \dots \lessdot x_k
= \hat 1_P$, we have $f(x_0,\hat 1) = 0$, $f(x_k,\hat 1) = 2$, and
$f(x_i,\hat 1) - f(x_{i-1},\hat 1) \in \{0,1\}$, so $f(x_i,\hat 1)$
must be $1$ for some $i$.  Therefore, the set $C_{f}=\{ x \in P \colon
f(x, \hat 1) = 1\}$ is a convex cut set contained in $P \setminus
\{\hat 0_P, \hat 1_P\}$, and for every rank $1$ element $y$ in
$A_2(j)$, the set $F_y = \{x \in C_{f} \colon f(x,y) = 1\}$ is an
order filter in $C_{f}$: if $x \in F_y$ and $x' \in C_f$, $x' \geq x$,
then $x' \in F_y$.

It remains to enumerate quilts $f$ in $\quilt(P,A_2(j))$ for which $\{
x \in P \colon f(x, \hat 1) = 1\}$ is equal to a given convex cut
subset $C$ of $P \setminus \{\hat 0_P, \hat 1_P\}$.  Given any choice
of an order filter $F_y$ in $C$ for each $y$ independently, one can
define a corresponding quilt $f \in \quilt(P,A_2(j))$ as follows. The
value $f(x,\hat 1)$ is $1$ if $x \in C$, $0$ if $x$ is below $C$, and
$2$ if $x$ is above $C$.  If $f(x,\hat 1)$ is $0$
or $2$, the other $f(x,y)$ are uniquely determined, and if $f(x,\hat
1) = 1$, the other $f(x,y)$ are determined by $F_y$. Since the order
filters of $C$ are in a natural bijection with the antichains in $C$,
this completes the proof of Equation~\eqref{eq:antichain}.

It is clear that if $C$ is a proper subset of $P \setminus \{\hat
0_P,\hat 1_P\}$, it contains strictly fewer antichains than $P
\setminus \{\hat 0_P,\hat 1_P\}$, so the term $d_1(P)^j$ coming from
$C = P \setminus \{\hat 0_P,\hat 1_P\}$ dominates.  This proves
\eqref{eq:antichain2}.
\end{proof}

\begin{example}\label{example:cut.set}
As a simple example, take $P = A_2(i)$ for $i \geq 1$. There is only
one (convex) cut set in $P \setminus \{\hat 0_P,\hat 1_P\}$, namely $P \setminus \{\hat 0_P,\hat 1_P\}$ itself. Every subset is an antichain, so $\alpha_P(P \setminus \{\hat 0_P,\hat 1_P\}) = 2^i$. Therefore $|\quilt(A_2(i),A_2(j))| = 2^{ij}$. This is consistent with Corollary~\ref{cor:powerofasm} for $n = 2$. In fact, it is easy to see that $\quilt(A_2(i),A_2(j)) \cong B_{ij}$ as lattices.
\end{example}

As a more involved application, let us show some enumerative results about quilts when $P$ is a chain or a product of chains, and $Q$ is an antichain. The $\ber_n$'s that appear in the corollary are the Bernoulli numbers $1,\frac{1}{2},\frac{1}{6},0,-\frac{1}{30},0,\frac{1}{42},\ldots$

\begin{cor} \label{prop:chainbyantichain}
 For arbitrary integers $j \geq 1$ and $k \geq 2$, we
have

\begin{equation}
 |\quilt(C_k,A_2(j))| = \sum_{i=2}^k (k+1-i)i^j.
 \end{equation}
Therefore,  $|\quilt(C_k,A_2(j))|$ as a function of $k$ is given by
the polynomial
\begin{equation}\label{eq:bernoulli}
 \frac{1}{(j+1)(j+2)} \left(k^{j+2} + \sum_{l = 1}^j \binom{j+2}l \left(l\ \ber_{l-1} - (l-1) \ber_l \right) k^{j+2-l} \right) + (\ber_j - \ber_{j+1} - 1) k.
\end{equation}
 For arbitrary $k$ we have
 $$|\quilt(C_{k} \times C_1,A_2(j))| = \sum_{\substack{0 \leq d \leq
 b,\, c \leq a \leq k \\ 1 \leq b \leq c+1 \leq k}} \left((a-c)(c-d+2)+{\textstyle \binom{c-d+3}2}-{\textstyle \binom{b-d+1}2}\right)^j.$$
 For arbitrary $k_1,k_2$, we have
 $$|\quilt(C_{k_1} \times C_{k_2},A_2(j))| \sim \left(\binom{k_1+k_2+2}{k_1+1}-2\right)^j.$$
\end{cor}
\begin{proof}
 A non-empty convex set in $C_k \setminus \{0,k\}$ is an interval $[i,j]$ for $1 \leq i \leq j \leq k-1$, and every such interval is a cut set. Furthermore, $[i,j]$ contains $|[i,j]]+1=j-i+2$ antichains---the empty set and all singletons. There are $k+1-i$ intervals of size $i-1$, $i \geq 2$, so \Cref{thm:antichain} gives $|\quilt(C_k,A_2(j))| = \sum_{i=2}^k (k+1-i)i^j = (k+1) \sum_{i=2}^k i^j - \sum_{i=2}^k i^{j+1}$. The famous Faulhaber's formula gives~\eqref{eq:bernoulli}.

 A subset $C$ of $C_{k} \times C_1 \setminus \{(0,0),(k,1)\}$ is a convex
 cut set if and only if it there exist $a,b,c,d$ so that $C = \{(i,0), b \leq i \leq a\} \cup \{(i,1), d \leq i \leq c\}$, where
 \begin{itemize}
  \item $b \geq 1$, otherwise $C$ contains $(0,0)$;
  \item $c \leq k-1$, otherwise $C$ contains $(k,1)$;
  \item $d \leq b$, otherwise the point $(b,1)$ is between $(b,0)$ and $(d,1)$ but not in $C$;
  \item $c \leq a$ for a similar reason; and
  \item $b \leq c+1$, otherwise there is a maximal chain that avoids $C$.
 \end{itemize}
 It is not hard to see that the number of antichains in such a set is $(a-c)(c-d+2)+{\textstyle \binom{c-d+3}2}-{\textstyle \binom{b-d+1}2}$. Now use \Cref{thm:antichain}.

 We can interpret an antichain in $C_{k_1} \times C_{k_2}$ as (the southwest corners of) a lattice path between $(k_1,0)$ and $(0,k_2)$, and there are $\binom{k_1+k_2+2}{k_1+1}$ of those. We subtract $2$ because we do not count the antichains $\{(0,0)\}$ and $\{(k_1,k_2)\}$. Again, use \Cref{thm:antichain}.
\end{proof}

\begin{example} The following illustrates the various statements of \Cref{prop:chainbyantichain}.
 We have
 $$|\quilt(C_4,A_2(j))| = 3 \cdot 2^j + 2 \cdot 3^j + 4^j$$
 for $j \geq 1$ and
 $$|\quilt(C_k,A_2(3))| = \frac{k^5}{20}+\frac{k^4}{4}+\frac{5 k^3}{12}+\frac{k^2}{4}-\frac{29 k}{30}$$
 for $k \geq 2$. The next statement gives
 $$|\quilt(C_2 \times C_1,A_2(j))| = 3^j + 2 \cdot 4^j + 2 \cdot 5^j + 2 \cdot 6^j + 8^j,$$
 and
 $$|\quilt(C_3 \times C_1,A_2(j))| = 2 \cdot 3^j + 3 \cdot 4^j + 4 \cdot 5^j + 5 \cdot 6^j + 2 \cdot 7^j + 4 \cdot 8^j + 3 \cdot 9^j + 2 \cdot 11^j + 13^j.$$
\end{example}

\begin{example}
It follows from Theorem \ref{thm:antichain} that $|\quilt(B_n,A_2(j))|
\sim d_1(B_n)^j$. It does not seem likely that a simple exact formula
for $|\quilt(B_n,A_2(j))|$ exists. Since $B_2 = A_2(2)$,
we have $|\quilt(B_2,A_2(j))| = 4^j$, and some (computer) time is needed to find
 $$|\quilt(B_3,A_2(j))| = 2 \cdot 8^j + 3 \cdot 9^j + 6 \cdot 10^j + 6 \cdot 13^j + 18^j.$$
The exact formula for $|\quilt(B_4,A_2(j))|$ can be found in 
Appendix A.
\end{example}

\section{Enumeration of chain quilts} \label{sec:chain}

Given a finite poset $P$, we will consider the enumeration of chain
quilts $\quilt(P,C_{n})$ in this section. The formulas are in terms
of \emph{fundamental} and \emph{standard} quilts for $P$.  We begin
with some additional notation and vocabulary.

Recall that we defined the sum of ranks $b(P) = \sum_{x \in P} \rank
x$. If $f \in \quilt(P,C_{b(P)})$, we say that $i \in [b(P)]$ is a
\emph{jump for $f$} if there exists $x \in P$ so that $f(x,i) =
f(x,i-1) + 1$. If the set of jumps of $f$ is equal to $[m]$, we say
that $f$ is \emph{$m$-fundamental for $P$}. A \emph{standard} quilt is
one that is $b(P)$-fundamental. Denote by $F_m(P)$ the set of all
$m$-fundamental quilts for $P$, and write $S(P) = F_{b(P)}$ and $F(P) =
\bigcup_m F_m(P)$.

A chain quilt is $m$-fundamental if and only if its MT form contains
precisely the elements $1,\ldots,m$. In particular, it is standard if
and only if its MT form contains exactly one of each of
$1,\ldots,b(P)$.  For example, consider $P = B_2$.  We have $b(P) =
4$, and there are four $2$-fundamental, five $3$-fundamental, and two
$4$-fundamental (standard) quilts, presented in
Figure~\ref{fig:fundamentalB2} in MT form.

 \begin{figure}[!ht]
 \begin{center}
\begin{tikzpicture}
\node (X0) at (0,0) {$\emptyset$};
\draw (-0.1,0.25) -- (-0.3,0.75);\draw (0.1,0.25) -- (0.3,0.75);
\node (X11) at (-0.4,1) {$1$}; \node (X12) at (0.4,1) {$1$\:};
\draw (-0.3,1.25) -- (-0.1,1.75);\draw (0.3,1.25) -- (0.1,1.75);
\node (X2) at (0,2) {$12$};
\end{tikzpicture}
\begin{tikzpicture}
\node (X0) at (0,0) {$\emptyset$};
\draw (-0.1,0.25) -- (-0.3,0.75);\draw (0.1,0.25) -- (0.3,0.75);
\node (X11) at (-0.4,1) {$1$}; \node (X12) at (0.4,1) {$2$\:};
\draw (-0.3,1.25) -- (-0.1,1.75);\draw (0.3,1.25) -- (0.1,1.75);
\node (X2) at (0,2) {$12$};
\end{tikzpicture}
\begin{tikzpicture}
\node (X0) at (0,0) {$\emptyset$};
\draw (-0.1,0.25) -- (-0.3,0.75);\draw (0.1,0.25) -- (0.3,0.75);
\node (X11) at (-0.4,1) {$2$}; \node (X12) at (0.4,1) {$1$\:};
\draw (-0.3,1.25) -- (-0.1,1.75);\draw (0.3,1.25) -- (0.1,1.75);
\node (X2) at (0,2) {$12$};
\end{tikzpicture}
\begin{tikzpicture}
\node (X0) at (0,0) {$\emptyset$};
\draw (-0.1,0.25) -- (-0.3,0.75);\draw (0.1,0.25) -- (0.3,0.75);
\node (X11) at (-0.4,1) {$2$}; \node (X12) at (0.4,1) {$2$\:};
\draw (-0.3,1.25) -- (-0.1,1.75);\draw (0.3,1.25) -- (0.1,1.75);
\node (X2) at (0,2) {$12$};
\end{tikzpicture}
\begin{tikzpicture}
\node (X0) at (0,0) {$\emptyset$};
\draw (-0.1,0.25) -- (-0.3,0.75);\draw (0.1,0.25) -- (0.3,0.75);
\node (X11) at (-0.4,1) {$1$}; \node (X12) at (0.4,1) {$2$\:};
\draw (-0.3,1.25) -- (-0.1,1.75);\draw (0.3,1.25) -- (0.1,1.75);
\node (X2) at (0,2) {$13$};
\end{tikzpicture}
\begin{tikzpicture}
\node (X0) at (0,0) {$\emptyset$};
\draw (-0.1,0.25) -- (-0.3,0.75);\draw (0.1,0.25) -- (0.3,0.75);
\node (X11) at (-0.4,1) {$2$}; \node (X12) at (0.4,1) {$1$\:};
\draw (-0.3,1.25) -- (-0.1,1.75);\draw (0.3,1.25) -- (0.1,1.75);
\node (X2) at (0,2) {$13$};
\end{tikzpicture}
\begin{tikzpicture}
\node (X0) at (0,0) {$\emptyset$};
\draw (-0.1,0.25) -- (-0.3,0.75);\draw (0.1,0.25) -- (0.3,0.75);
\node (X11) at (-0.4,1) {$2$}; \node (X12) at (0.4,1) {$2$\:};
\draw (-0.3,1.25) -- (-0.1,1.75);\draw (0.3,1.25) -- (0.1,1.75);
\node (X2) at (0,2) {$13$};
\end{tikzpicture}
\begin{tikzpicture}
\node (X0) at (0,0) {$\emptyset$};
\draw (-0.1,0.25) -- (-0.3,0.75);\draw (0.1,0.25) -- (0.3,0.75);
\node (X11) at (-0.4,1) {$2$}; \node (X12) at (0.4,1) {$3$\:};
\draw (-0.3,1.25) -- (-0.1,1.75);\draw (0.3,1.25) -- (0.1,1.75);
\node (X2) at (0,2) {$13$};
\end{tikzpicture}
\begin{tikzpicture}
\node (X0) at (0,0) {$\emptyset$};
\draw (-0.1,0.25) -- (-0.3,0.75);\draw (0.1,0.25) -- (0.3,0.75);
\node (X11) at (-0.4,1) {$3$}; \node (X12) at (0.4,1) {$2$\:};
\draw (-0.3,1.25) -- (-0.1,1.75);\draw (0.3,1.25) -- (0.1,1.75);
\node (X2) at (0,2) {$13$};
\end{tikzpicture}
\begin{tikzpicture}
\node (X0) at (0,0) {$\emptyset$};
\draw (-0.1,0.25) -- (-0.3,0.75);\draw (0.1,0.25) -- (0.3,0.75);
\node (X11) at (-0.4,1) {$2$}; \node (X12) at (0.4,1) {$3$\:};
\draw (-0.3,1.25) -- (-0.1,1.75);\draw (0.3,1.25) -- (0.1,1.75);
\node (X2) at (0,2) {$14$};
\end{tikzpicture}
\begin{tikzpicture}
\node (X0) at (0,0) {$\emptyset$};
\draw (-0.1,0.25) -- (-0.3,0.75);\draw (0.1,0.25) -- (0.3,0.75);
\node (X11) at (-0.4,1) {$3$}; \node (X12) at (0.4,1) {$2$};
\draw (-0.3,1.25) -- (-0.1,1.75);\draw (0.3,1.25) -- (0.1,1.75);
\node (X2) at (0,2) {$14$};
\end{tikzpicture}
\end{center}
\caption{All fundamental quilts for $B_2$.} \label{fig:fundamentalB2}
\end{figure}
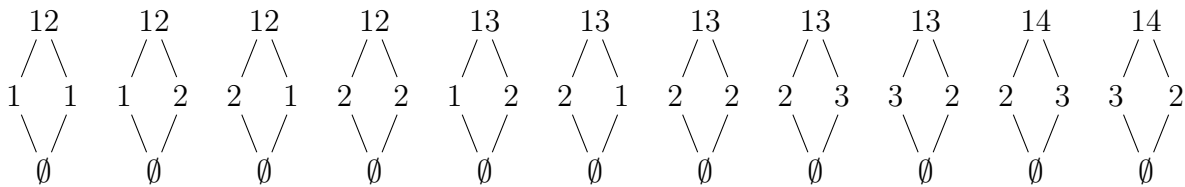

Observe that $\rank P \leq b(P)$ and a quilt can be $m$-fundamental
only for $ \rank P \leq m \leq b(P)$.  Note too that $S(P)$ is not
empty. Indeed, to find a standard chain quilt for $P$, consider the
map that sends each $x \in P$ of rank $i$ to the subset $\{1,2,\ldots,
i \}$, and then ``standardize'': change all the $1$'s in the reverse
of the chosen total order on $P$ to $1,2,\ldots,j_1$, then change all
of the original $2$'s to $j_1+1,j_1+2,\ldots,j_1+j_2$, then all
original $3$'s to $j_1+j_2+1,j_1+j_2+2,\ldots,j_1+j_2+j_3$
etc. Figure~\ref{fig:standardB3} shows the standard quilt we get for
$B_3$. There are $1344$ standard chain quilts for $B_{3}$ in total.

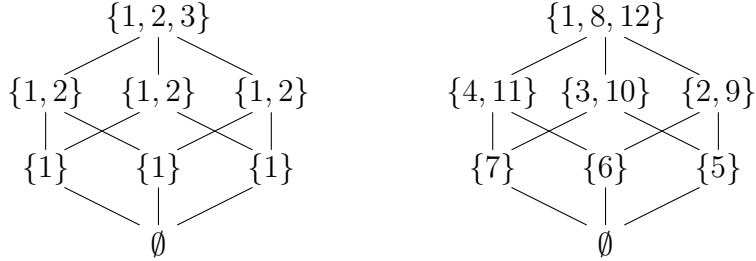
\begin{figure}[!ht]
\begin{center}
\begin{tikzpicture}
\node (X0) at (0,0) {$\emptyset$};
\draw (-0.25,0.25) -- (-1.25,0.75);\draw (0,0.25) -- (0,0.75);\draw (0.25,0.25) -- (1.25,0.75);
\node (X11) at (-1.5,1) {$\{1\}$}; \node (X12) at (0,1) {$\{1\}$}; \node (X13) at (1.5,1) {$\{1\}$};
\draw (-1.5,1.25) -- (-1.5,1.75);\draw (-1.25,1.25) -- (-0.25,1.75);\draw (-0.25,1.25) -- (-1.25,1.75);\draw (0.25,1.25) -- (1.25,1.75);\draw (1.25,1.25) -- (0.25,1.75);\draw (1.5,1.25) -- (1.5,1.75);
\node (X21) at (-1.5,2) {$\{1,2\}$}; \node (X12) at (0,2) {$\{1,2\}$}; \node (X13) at (1.5,2) {$\{1,2\}$};
\draw (-0.25,2.75) -- (-1.25,2.25);\draw (0,2.75) -- (0,2.25);\draw (0.25,2.75) -- (1.25,2.25);
\node (X3) at (0,3) {$\{1,2,3\}$};
\end{tikzpicture}
\hspace{.5in}
\begin{tikzpicture}
\node (X0) at (0,0) {$\emptyset$};
\draw (-0.25,0.25) -- (-1.25,0.75);\draw (0,0.25) -- (0,0.75);\draw (0.25,0.25) -- (1.25,0.75);
\node (X11) at (-1.5,1) {$\{7\}$}; \node (X12) at (0,1) {$\{6\}$}; \node (X13) at (1.5,1) {$\{5\}$};
\draw (-1.5,1.25) -- (-1.5,1.75);\draw (-1.25,1.25) -- (-0.25,1.75);\draw (-0.25,1.25) -- (-1.25,1.75);\draw (0.25,1.25) -- (1.25,1.75);\draw (1.25,1.25) -- (0.25,1.75);\draw (1.5,1.25) -- (1.5,1.75);
\node (X21) at (-1.5,2) {$\{4,11\}$}; \node (X12) at (0,2) {$\{3,10\}$}; \node (X13) at (1.5,2) {$\{2,9\}$};
\draw (-0.25,2.75) -- (-1.25,2.25);\draw (0,2.75) -- (0,2.25);\draw (0.25,2.75) -- (1.25,2.25);
\node (X3) at (0,3) {$\{1,8,12\}$};
\end{tikzpicture}
\caption{A 3-fundamental quilt for $B_{3}$ along with
its standardization.}  \label{fig:standardB3}
\end{center}
\end{figure}

\begin{theorem}  \label{thm:chainenumeration}
 For a fixed poset $P$ of rank $k$ with least and greatest elements
and any integer $n \geq k$, the number of chain quilts of type $(P,
C_{n})$ is given by a polynomial in $n$, namely

\begin{equation} \label{eq:chainenumeration}
|\quilt(P, C_n)| = \sum_{m = k}^{b(P)} |F_m(P)| \binom n
m.  \end{equation} In particular, \begin{equation}
\label{eq:chainasymptotics} |\quilt(P, C_n)| \sim \frac{|S(P)|}{b(P)!}
\cdot n^{b(P)}.  \end{equation}
\end{theorem}

\begin{proof}
We will abuse notation and consider a quilt in $ f \in \quilt(P,
C_n)$, with $n \geq k$, to be synonymous with its jump set map $f
\colon P \to B_n$ for which $|f(x)| = \rank x$ for all $x \in P$ and
for which $f(x)$ and $f(y)$ interlace when $x \lessdot y$, see
\Cref{prop:interlacing}. Say that $\bigcup_{x \in P} f(x) =
\{i_1,\dots,i_m\} \subseteq [n]$. Replace all instances of $i_j$ with
$j$ for all $j \in [m]$. This gives us an $m$-fundamental quilt, and for every
$m$-fundamental quilt, there are $\binom n m$ ways to choose the map
$j \mapsto i_j$. This proves Equation~\eqref{eq:chainenumeration}. The
highest degree terms are clearly the ones with $m = b(P)$, which
implies Equation~\eqref{eq:chainasymptotics}.
\end{proof}

To illustrate the procedure employed in the proof, take the chain
quilt on the left of Figure~\ref{fig:fundamentalB3}. The union of all
the jump sets is $\{2,4,6,7,9\}$, so we replace $2$, $4$, $6$, $7$, $9$ by $1$, $2$, $3$, $4$, $5$, respectively. We get the $5$-fundamental quilt on the right.

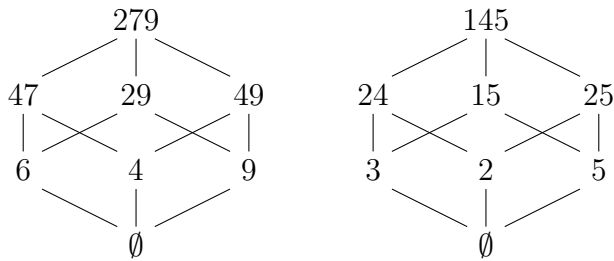
\begin{figure}[!ht]
\begin{center}
\begin{tikzpicture}
\node (X0) at (0,0) {$\emptyset$};
\draw (-0.25,0.25) -- (-1.25,0.75);\draw (0,0.25) -- (0,0.75);\draw (0.25,0.25) -- (1.25,0.75);
\node (X11) at (-1.5,1) {$6$}; \node (X12) at (0,1) {$4$}; \node (X13) at (1.5,1) {$9$};
\draw (-1.5,1.25) -- (-1.5,1.75);\draw (-1.25,1.25) -- (-0.25,1.75);\draw (-0.25,1.25) -- (-1.25,1.75);\draw (0.25,1.25) -- (1.25,1.75);\draw (1.25,1.25) -- (0.25,1.75);\draw (1.5,1.25) -- (1.5,1.75);
\node (X21) at (-1.5,2) {$47$}; \node (X12) at (0,2) {$29$}; \node (X13) at (1.5,2) {$49$};
\draw (-0.25,2.75) -- (-1.25,2.25);\draw (0,2.75) -- (0,2.25);\draw (0.25,2.75) -- (1.25,2.25);
\node (X3) at (0,3) {$279$};
\end{tikzpicture} \qquad
\begin{tikzpicture}
\node (X0) at (0,0) {$\emptyset$};
\draw (-0.25,0.25) -- (-1.25,0.75);\draw (0,0.25) -- (0,0.75);\draw (0.25,0.25) -- (1.25,0.75);
\node (X11) at (-1.5,1) {$3$}; \node (X12) at (0,1) {$2$}; \node (X13) at (1.5,1) {$5$};
\draw (-1.5,1.25) -- (-1.5,1.75);\draw (-1.25,1.25) -- (-0.25,1.75);\draw (-0.25,1.25) -- (-1.25,1.75);\draw (0.25,1.25) -- (1.25,1.75);\draw (1.25,1.25) -- (0.25,1.75);\draw (1.5,1.25) -- (1.5,1.75);
\node (X21) at (-1.5,2) {$24$}; \node (X12) at (0,2) {$15$}; \node (X13) at (1.5,2) {$25$};
\draw (-0.25,2.75) -- (-1.25,2.25);\draw (0,2.75) -- (0,2.25);\draw (0.25,2.75) -- (1.25,2.25);
\node (X3) at (0,3) {$145$};
\end{tikzpicture}
\end{center}
\caption{A chain quilt and a $5$-fundamental quilt for $B_3$.} \label{fig:fundamentalB3}
\end{figure}

\begin{example}
 From the previous example it follows that
 $$|\quilt(B_2,C_n)| = 4 \binom n 2 + 5 \binom n 3 + 2 \binom n 4 =
 \frac{n^4}{12}+\frac{n^3}{3}+\frac{5 n^2}{12}-\frac{5 n}{6}
\sim 2 \cdot \frac{n^4}{4!}
$$
 for $n \geq 2$. This agrees with \Cref{prop:chainbyantichain} (note that $B_2 = A_2(2)=C_1 \times C_1$). With more effort, we can compute
 $$|\quilt(B_3,C_n)| \sim \frac{1344}{12!} \cdot n^{12} \quad \text{and} \quad \quilt(B_4,C_n)| \sim \frac{10651644896477184}{32!} \cdot n^{32},$$
 see Appendix A.
\end{example}

\begin{example} \label{ex:rectangularASMs}
 We can think of a standard quilt for $P = C_k$ as a monotone triangle
 (in the classical sense) in which all numbers $1,\ldots,\binom{k+1}2$
 appear. After an up-down reflection and a $45^\circ$ rotation, we get
 a shifted standard Young tableau of staircase shape $(k,k-1,\ldots,1)$. For example, for $k=3$, we get monotone triangles
 $$\begin{array}{ccccc} 1&&3&&6 \\ &2&&5& \\ &&4&& \end{array} \quad \mbox{and} \quad \begin{array}{ccccc} 1&&4&&6 \\ &2&&5& \\ &&3&& \end{array}$$
 and shifted standard Young tableaux
 $$\begin{array}{ccc} 1 & 2 & 4 \\ & 3 & 5 \\ & & 6 \end{array} \quad \mbox{and} \quad \begin{array}{ccc} 1 & 2 & 3 \\ & 4 & 5 \\ & & 6 \end{array}.$$
 The hook-length formula for shifted standard Young tableaux \cite{Thr} gives (for fixed $k$ and $n \to \infty$)
$$|\ASM_{k,n}| \sim \frac{\prod_{i=0}^{k-1} (2i)!}{\prod_{i=0}^{k-1} (k + i)!} \cdot n^{\binom{k+1}2}.$$
 An $m$-fundamental quilt in this case can be interpreted as a shifted tableau of staircase shape $(k,k-1,\ldots,1)$ with weakly increasing rows and columns, strictly increasing diagonals, and entries in $[m]$, with each number in $[m]$ appearing at least once.
\end{example}

\begin{remark}
 Every $m$-fundamental quilt can be obtained from a standard quilt by replacing $1,\ldots,b(P)$ by a sequence of the type $1,\ldots,1,2,\ldots,2,\ldots,m,\ldots,m$, with each number appearing at least once. The only requirement when choosing the replacing sequence is that if $i,j$ appear in the same set of the standard quilt, $i$ and $j$ cannot be replaced by the same number. However, one $m$-fundamental quilt can be obtained from several standard ones. For example, for $P = B_2$, the replacing sequence $1223$ on both fundamental quilts (the last two in Figure~\ref{fig:fundamentalB2}) gives the same $3$-fundamental quilt (the seventh quilt in Figure~\ref{fig:fundamentalB2}). That means that we can rephrase Equation~\eqref{eq:chainenumeration} as
 $$|\quilt(P, C_n)| = \sum_{f \in S(T)} \sum_{u} \binom n {\max u},$$
 where $u$ runs over integer sequences that are ``compatible'' with $f$, where the definition of compatibility ensures that there are no repetitions of fundamental quilts. We omit the details.
\end{remark}

There is in fact one more way to compute $|\quilt(P,C_n)|$ and prove
the polynomiality property via the  transfer-matrix method
\cite[Thm. 4.7.2]{ec1}  using the adjacency matrix $A_D(P)$ of the
Dedekind graph of $P$ defined in \Cref{def:Dedekind.graphs}. While it
does not imply Equation~\eqref{eq:chainasymptotics}, it is the
authors' experience that it is in practice easier to compute the
inverse of the (upper-triangular) matrix $I - x A_D(P)$ than the
cardinalities of $F_m(P)$ for $m = k,\ldots,b(P)$. Furthermore, this
method also gives us a way to compute $|\quilt(P, C_1)|,\ldots,
|\quilt(P, C_{k-1})|$.

\begin{theorem} \label{thm:graph}
 For a finite poset $P$ of rank $k\geq 1$ with least and greatest elements,  we have
\begin{equation}\label{eq:transfer.1}
\sum_{n=k}^\infty |\quilt(P, C_n)|  x^n = (I - x
 A_D(P))^{-1}_{1,d(P)} = \frac{(-1)^{d(P)-1}}{(1-x)^{d(P)}} \det T(P),
\end{equation}
 where $T(P)$ is the transfer-matrix $I - x A_D(P)$ with the first column
 and last row removed. In particular, the sequence
 $0,0,\ldots,|\quilt(P, C_k)|, |\quilt(P, C_{k+1})|,\ldots$ is given
 by a polynomial of degree $< d(P)$. 
 Furthermore,
\begin{equation}\label{eq:transfer.2}
\sum_{n=0}^{k-1} |\quilt(P, C_n)| x^n = \sum_{i=1}^{d(P)-1} (I - x A'_D(P))^{-1}_{1,i} =  \sum_{i=1}^{d(P)-1} (-1)^{i-1} \det T'(P)_i,
\end{equation}
  where $T'(P)_i$ is the matrix $I - x A'_D(P)$ with the first column and $i$-th row removed.
\end{theorem}

\begin{proof}
A chain quilt $f \in \quilt(P,C_n)$ can be viewed as a sequence of
Dedekind maps, with the $i$-th one, $0 \leq i \leq n$, sending $x$ to
$f(x,i)$. The map $x \mapsto f(x,0)$ is always the zero map, which is
the only element in $D_0(P)$. If $n \geq \rank P = k$, $f(x,n)=\rank
x$ is the only element in $D_k(P)$. Furthermore, there is an edge from
$x \mapsto f(x,i-1)$ to $x \mapsto f(x,i)$ in the Dedekind graph
$G_D(P)$. In other words, we can interpret a chain quilt as a walk on
the graph $G_D(P)$ starting in the first vertex (the all zero map) and
ending in the last vertex (the only element in $D_k(P)$). The
transfer-matrix method \cite[Thm. 4.7.2]{ec1} tells us that the
generating function for such walks is the $(1,d(P))$ entry of the
matrix $(I - x A_D(P))^{-1}$, or, equivalently, is given by the
corresponding determinantal expression, which proves
\eqref{eq:transfer.1}. By definition, $\det T(P)$ is a polynomial of
degree $< d(P)$, which implies $(-1)^{d(P)-1} \det T(P) (1-x)^{-d(P)}$
is a rational function of $x$ of degree $< 0$.  Hence, its coefficients
as a power series, namely $|\quilt(P, C_n)|$, are given by a
polynomial function of $n$.

If $n < \rank P$, we have $f(\hat 1_P,i) = i$ for every $i
\in [n]$. In other words, given $f \in \quilt(P,C_n)$ and $i\in [n]$
there is an edge from $x \mapsto f(x,i-1)$ to $x \mapsto f(x,i)$ in
the restricted Dedekind graph $G_D'(P)$, and we are looking at walks
on the graph $D'(P)$ starting in the first vertex and ending anywhere
except in the last vertex. This proves \eqref{eq:transfer.2}.
\end{proof}

Recall from \Cref{sub:asm.fubini} that there is an easy bijection
between $k \times n$ ASMs and monotone triangles with all possible
length $k$ top row sequences.  Such a top row sequence will be denoted
by $(a_1,\ldots,a_k)$ with $1 \leq a_1 < a_2 < \dots < a_k \leq
n$. Fischer proved that the cardinality of $\MT(a_1,\ldots,a_k)$, the
set of monotone triangles with top row $(a_1,\ldots,a_k)$ is a
polynomial in variables $a_1,\ldots,a_k$, and she also found an
explicit (operator) formula for $|\MT(a_1,\ldots,a_k)|$, see
\cite{MTs}. The definition can be extended to arbitrary chain quilts:
given a poset $P$ of rank $k$ and $1 \leq a_1 < a_2 < \dots < a_k \leq
n$, define $\MT_P(a_1,\ldots,a_k)$ as the set of quilts $f \in
\quilt(P,C_n)$ for which $J_f(\hat 1_P) = \{a_1,\ldots,a_k\}$.  Here, we
equate the quilt $f$ with its jump set map $f:P \longrightarrow B_{n}$
as in the proof of \Cref{thm:chainenumeration}.  We call $J_f(\hat 1_P)$
the \emph{top set of the quilt $f$}. Note that strictly speaking,
$\MT(a_1,\ldots,a_k)$ depends on $n$, but there is a natural bijection
between $\MT(a_1,\ldots,a_k) \subseteq \quilt(P,C_n)$ and
$\MT(a_1,\ldots,a_k) \subseteq \quilt(P,C_m)$ whenever $m,n \geq a_k$
so we can ignore this.   Let $J_f(\hat 1_P)_i$ denote the
$i$-th largest element of the set $J_f(\hat 1_P)$.

\begin{theorem} \label{thm:mt}
 For a finite poset $P$ of rank $k$ with least and greatest elements,
we have \begin{equation} \label{eq:mt} |\MT_P(a_1,\ldots,a_k)| =
\sum_{f \in F(P)} \prod_{i=2}^k \binom{a_i-a_{i-1}-1}{J_f(\hat 1_P)_i -
J_f(\hat 1_P)_{i-1}-1}. \end{equation} 
\end{theorem}

\begin{proof}
 This is similar to the proof of
Theorem~\ref{thm:chainenumeration}. Given a chain quilt $g \in
\quilt(P,C_n)$ with top set $\{a_1,\ldots,a_k \}$, let 
$\{i_1<\dots<i_m\} = \bigcup_{x \in P} g(x) \subseteq [n]$. 
Replace all instances of $i_j$ with $j$ to get an $m$-fundamental
quilt $f_{g} \in F(P)$.  Each quilt in the inverse image under the
replacement map of an $m$-fundamental quilt $ f \in F(P)$ with
$J_f(\hat 1_{P})=\{j_1,\ldots,j_k \}$ is determined by
replacing $j_{i}$  by $a_{i}$ everywhere in the MT form of $f$ for each
$1\leq i\leq k$, making 
a choice of
$j_2-j_1-1$ elements among $a_1+1,\ldots,a_2-1$ to replace
$j_{1}+1,\dotsc , j_{2}-1$, making a choice of
$j_3-j_2-1$ elements among $a_2+1,\ldots,a_3-1$ etc., which proves
Equation~\eqref{eq:mt}.
\end{proof}

To illustrate the proof, say that we want to enumerate quilts in
$\quilt(B_3,C_{20})$ with top set $(2,10,16)$.  We can get all such
quilts from $m$-fundamental quilts for $3 \leq m \leq 12$. For
example, we can take the $5$-fundamental quilt on the right in
Figure~\ref{fig:fundamentalB3} with top set $\{1,4,5 \}$, and replace $1$ by $2$, $4$ by $10$
and $5$ by $16$ to get the correct top set. We have choices for what
we replace $2$ and $3$ by: we can select any $2$ of the elements
between $3$ and $9$ for that, and there are $\binom{7}{2}$ ways to do
that.
Since $4,5$ are adjacent values there are no further choices to make in
this case.  
On the other hand, if we take a, say, standard quilt with top set
$(1,8,12)$, like the one in Figure~\ref{fig:standardB3} we replace $1$
by $2$, $8$ by $10$ and $12$ by $16$, and we select any $6$ elements
between $3$ and $9$ to replace $2,\ldots,7$ by, and any $3$ elements
between $11$ and $15$ to replace $9,10,11$ by. Therefore we have
$\binom 7 6 \cdot \binom 5 3$ choices.  Using \Cref{thm:mt}, one can
compute $|\MT_{B_{3}}(2,10,16)|=52202240$.

\begin{example}
 From Figure~\ref{fig:fundamentalB2}, we get that
 $$|\MT_{B_2}(a_1,a_2)| = 4 + 5 \binom{a_2-a_1-1}{3-1-1} + 2
 \binom{a_2-a_1-1}{4-1-1} = (a_2-a_1+1)^2.$$
By \Cref{prop:interlacing}, one can observe 
that the following general formula holds:
\[
|\MT_{A_2(j)}(a_1,a_2)| = (a_2-a_1+1)^j.
\]
See Appendix A for the much less obvious expression for $|\MT_{B_3}(a_1,a_2,a_3)|$.
\end{example}

We conclude the section with the following observation about the
$k$-fundamental quilts for a poset of rank $k$.  These quilts are the
most compressed fundamental quilts for such a poset.

\begin{cor}\label{cor:k.k-1}
Assume $P$ has rank $k$.  Then we have
\begin{equation}\label{eq:k.cor}
|F_k(P)| = |\MT_P(1,\ldots,k)| = |\quilt (P, C_{k})| = |\quilt (P, C_{k-1})|.
\end{equation}
\end{cor}
\begin{proof}
The first equality holds by \Cref{thm:mt}.  Take a
quilt $f \in \quilt(P,C_k)$. Since $\rank P = \rank C_k$,
\Cref{lem:chains.bounds} gives $f(x,k) = \rank x$ for all $x \in
P$. This means that the map $\quilt (P, C_k) \longrightarrow \quilt
(P, C_{k-1})$ defined by $f \mapsto f|_{P \times C_{k-1}}$ is an
isomorphism of lattices proving the last equality.
\Cref{lem:chains.bounds} also says that the top set $J_f(\hat 1_P)$ is
$[k]$, which proves $|\quilt (P, C_{k})| = |\MT_P(1,\ldots,k)|$.
\end{proof}

\section{Enumeration of Boolean quilts} \label{sec:boolean}

Exact enumeration of Dedekind maps for $B_{n}$ and Boolean quilts is at least as difficult as
finding a formula for the Dedekind numbers. However, some bounds can
be given. For example, we can construct $2^{\binom n{\floor{n/2}}}$
$1$-Dedekind maps on $B_{n}$ by taking $f(T) = 0$ for $|T| <
\floor{n/2}$, $f(T) = 1$ for $|T| > \floor{n/2}$, $f(T) \in \{0,1\}$
for $|T| = \floor{n/2}$. It follows that $d_1(B_n) \geq 2^{\binom
n{\floor{n/2}}}$.  In 1966, Hansel proved that $d_1(B_n) \leq
3^{\binom n{\floor{n/2}}}$ \cite{hansel}. Kleitman \cite{kleitman} improved that to
\begin{equation}\label{eq:kleitman}
d_1(B_n) \leq 2^{(1+c \ln n/\sqrt n)\binom{n}{\floor{n/2}} }
\end{equation}
for some constant $c$. We will use that result in the following.

\begin{lemma} \label{lem:booleandedekind}
 There exists a constant $c > 0$ so that for all $1\leq k\leq n$,
 $$d_k(B_n) \leq 2^{k(1+c \ln n/\sqrt n)\binom{n}{\floor{n/2}} }.$$
 Furthermore, for every $\varepsilon > 0$,
 $$d_k(B_n) \geq 2^{(k-\varepsilon) \binom n {\floor{n/2}}}$$
 for large enough $n$.
\end{lemma}

\begin{proof}
 The upper bound follows from \Cref{lemma:dedekindupper} and \eqref{eq:kleitman}. To prove the lower bound,  assume that $n \geq 2k$. Consider the set
of maps on $B_{n}$ defined by the following criteria:
 \begin{itemize}
  \item $f(T) = 0$ if $|T| \leq \floor{n/2} - 2 \floor{k/2}$;
  \item $f(T) \in \{0,1\}$ if $|T| = \floor{n/2} - 2 \floor{k/2} + 1$;
  \item $f(T) = 1$ if $|T| = \floor{n/2} - 2 \floor{k/2} + 2$;
  \item $f(T) \in \{1,2\}$ if $|T| = \floor{n/2} - 2 \floor{k/2} + 3$;
  \item[$\vdots$]
  \item $f(T) \in \{\floor{k/2}-1,\floor{k/2}\}$ if $|T| = \floor{n/2} - 1$;
  \item $f(T) = \floor{k/2}$ if $|T| = \floor{n/2}$;
  \item $f(T) \in \{\floor{k/2},\floor{k/2}+1\}$ if $|T| = \floor{n/2} + 1$;
  \item[$\vdots$]
  \item $f(T) \in \{k-1,k\}$ if $|T| = \floor{n/2} + 2 \floor{k/2} - 1$;
  \item $f(T) = k$ if $|T| \geq \floor{n/2} + 2 \floor{k/2}$.
 \end{itemize}
 Every such map is a $k$-Dedekind map. It follows that
 $$d_k(B_n) \geq 2^{\sum_{i=-\floor{k/2}+1}^{\ceiling{k/2}} \binom{n}{\floor{n/2}+2i-1}}.$$
 For a fixed $k$ and for $-k \leq i \leq k$,
 $$\lim_{n \to \infty} \frac{\binom{n}{\floor{n/2}+i}}{\binom{n}{\floor{n/2}}} = 1.$$
 That means that for a chosen $\varepsilon > 0$, we have $\binom{n}{\floor{n/2}+2i-1} \geq (1-\varepsilon/k) \binom{n}{\floor{n/2}}$ for $i=-\floor{k/2}+1,\ldots,\ceiling{k/2}$ for $n$ large enough. The second statement of the lemma now follows.
\end{proof}

\begin{theorem}\label{thm:bounds}
Let $P$ be a finite ranked poset with least and greatest
elements. There exists a constant $c>0$ so that if $n \geq \rank P$, then
$$2^{\binom n{\floor{n/2}}} \leq |\quilt(P,B_n)| \leq
2^{b(P) (1+c \ln n/\sqrt n)\binom n{\floor{n/2}}}.$$  If $n \geq 2
\rank P$, we have the improved lower bound $$|\quilt(P,B_n)| \geq d_1(P)^{\binom
n{\floor{n/2}}}.
$$
In particular, $$ 2^{\binom
k{\floor{k/2}} \binom n{\floor{n/2}}}\leq |\quilt(B_k,B_n)| \leq
2^{k2^{k-1} (1+c \ln n/\sqrt n)\binom n{\floor{n/2}}} $$
for $n \geq 2k$.
\end{theorem}
\begin{proof}
By \Cref{thm:upperbound}, we have
$$|\quilt(P,B_n)| \leq d_1(B_n)^{b(P)},$$
and the upper bound for $|\quilt(P,B_n)|$ now follows from \eqref{eq:kleitman}.

The $2^{\binom n{\floor{n/2}}}$ quilts of type $(P, B_{n})$ defined by
$$f(x, T) =
\left\{ \begin{array}{ccl}
 \min\{\rank x,\ |T|\} & : & |T| < \floor{k/2} \\
 \min\{\rank x,\ \floor{k/2}\} & : & \floor{k/2} \leq |T| < \floor{n/2} \\
 \min\{\rank x,\ \floor{k/2}+\epsilon_T\} & : & |T| = \floor{n/2} \\
 \min\{\rank x,\ \floor{k/2}+1\} & : & \floor{n/2} < |T| \leq n - \ceiling{k/2}  \\
 \min\{\rank x,\ |T|-n+k\} & : & n - \ceiling{k/2} <  |T|
\end{array} \right., $$
where $\epsilon_T \in \{0,1\}$, prove the first lower bound.

For the second lower bound, assume $n \geq 2k$.
For each $T \subseteq [n]$ with $|T| = \floor{n/2}$, choose a
1-Dedekind map $g_T \in D_1(P)$. Then, each such collection of choices
determines a distinct quilt of type $(P,B_n)$ given by
\[
f(x,T) = \begin{cases}
0 &  |T| < \floor{n/2} \\
g_T(x) &  |T| = \floor{n/2} \\
\min\{\rank x,\ |T|-\floor{n/2},k\} &  |T| > \floor{n/2}.
\end{cases}
\]
It follows that $|\quilt(P,B_n)| \geq d_1(P)^{\binom
 n{\floor{n/2}}}$. The last inequality follows from $d_1(B_k) \geq
 2^{\binom k{\floor{k/2}}}$ from the beginning of this section.
\end{proof}

\begin{remark}
\Cref{thm:bounds}  guarantees that for a poset $P$, there are positive numbers $A_P$ and $B_P$ such that
 $$\frac{\ln |\quilt(P,B_n)|}{\binom{n}{\floor{n/2}}} \in [A_P,B_P]$$
 for $n \geq \rank P$.  It is natural to ask if the limit
 $$L(P) = \lim_{n \to \infty} \frac{\ln |\quilt(P,B_n)|}{\binom{n}{\floor{n/2}}}$$
 exists. By the last part of the theorem, if $L(B_k)$ exists, it must be in the interval
 $$\left[\binom k{\floor{k/2}} \ln 2, k2^{k-1} \ln 2\right].$$
We do not have enough data to state an explicit conjecture, but we believe that the limit does indeed exist; if we had to venture a guess as to what this number would be, we would say $L(P) = b(P) \ln 2$. In other words, we believe that $2^{b(P) \binom n{\floor{n/2}}}$ is the best estimate for $|\quilt(P,B_n)|$ among functions of the form $C^{\binom n{\floor{n/2}}}$.
\end{remark}

\section{Final remarks} \label{sec:final}

\subsection*{Representability}

Call a quilt $f \in \quilt(B_k, B_n)$ \emph{representable} if there
exists a matrix $A \in \R^{k \times n}$, $\rank A = \min\{k,n\}$, so
that $f(I,J)$ is equal to the rank of the matrix obtained by taking
rows in $I$ and columns in $J$ in the matrix $A$. For $n = k = 2$,
there are $7$ representable quilts $f_1,\ldots,f_7$ coming from, say,
matrices $\left[ \begin{smallmatrix} 1 & 0 \\ 0 & 1 \end{smallmatrix}
\right]$, $\left[ \begin{smallmatrix} 0 & 1 \\ 1 & 0 \end{smallmatrix}
\right]$, $\left[ \begin{smallmatrix} 1 & 1 \\ 0 & 1 \end{smallmatrix}
\right]$, $\left[ \begin{smallmatrix} 1 & 0 \\ 1 & 1 \end{smallmatrix}
\right]$, $\left[ \begin{smallmatrix} 1 & 1 \\ 1 & 0 \end{smallmatrix}
\right]$, $\left[ \begin{smallmatrix} 0 & 1 \\ 1 & 1 \end{smallmatrix}
\right]$, $\left[ \begin{smallmatrix} 1 & 1 \\ -1 &
1 \end{smallmatrix} \right]$.  The lattice $\quilt(B_2, B_2) \cong
B_4$, on the other hand, contains $16$ elements, so they are not all representable.  It would be
interesting to understand representable quilts better.

\begin{open}\label{open:boolean.realizability}
Characterize the representable quilts of type $(B_{k}, B_{n})$.
\end{open}

\begin{open}\label{open:flag.realizability}
Characterize the representable chain quilts of types $(C_{k}, B_{n})$ and
those which correspond to $f_{w}$ for some $w \in W_{n,k}$.  
\end{open}

For comparison, all chain quilts of types $(C_{k}, C_{n})$ are
representable.  In fact, they can be realized as the rank functions of
partial permutation matrices \cite[Lem. 3.1]{Fulton1}. This fact was
important in Fulton's definition of the essential set of a
permutation.  We note that the referee asked if one can characterize
the analog of the essential set for all quilts.  This is currently an
open problem.

\subsection*{Dedekind--MacNeille completion}

It is well known that the lattice of alternating sign matrices is the
Dedekind--MacNeille completion of (i.e., the smallest lattice
containing) the strong Bruhat order on $S_n$. A natural question is
whether the lattice $\quilt(B_k, B_n)$ is the Dedekind--MacNeille
completion of the poset of representable quilts. The answer, however,
is no. The poset and the completion are shown in \Cref{fig:DMC}.  The
following problem inspired the exploration of quilts as a
generalization of ASMs.  However, it remains open.

\begin{open}\label{open:flag.completion}(Posed by Stark Ryan and
independently by Jessica Striker) What is the Dedekind--MacNeille
completion of the medium roast partial order on Fubini words in
$W_{n,k}$ defined in \Cref{sec:motivation}?  By
\Cref{cor:medium.roast.quilts}, this complete lattice must be
isomorphic to a sublattice of $\quilt(C_{k},B_{n})$ and contain the
quilts of the form $f_{w}$ for $w \in W_{n,k}$.
\end{open}

\begin{open}\label{open:bbcompletion}
Find the Dedekind--MacNeille completion of the poset of representable
quilts of type $(B_{k},B_{n})$.
\end{open}

\begin{figure}[!ht]
\begin{center}
\begin{tikzpicture}
\node at (0,-1){$\phantom{g_0}$};
\node at (-1,0) {$f_1$};
\node at (1,0) {$f_2$};
\node at (-1.5,1) {$f_3$};
\node at (-0.5,1) {$f_4$};
\node at (0.5,1) {$f_5$};
\node at (1.5,1) {$f_6$};
\node at (0,2) {$f_7$};
\draw (-1.125,0.25) -- (-1.375,0.75);
\draw (-0.875,0.25) -- (-0.625,0.75);
\draw (0.875,0.25) -- (0.625,0.75);
\draw (1.125,0.25) -- (1.375,0.75);
\draw (-1.25,1.25) -- (-0.25,1.75);
\draw (-0.375,1.25) -- (-0.125,1.75);
\draw (0.375,1.25) -- (0.125,1.75);
\draw (1.25,1.25) -- (0.25,1.75);
\end{tikzpicture} \qquad
\begin{tikzpicture}
\node at (0,-1) {$g_0$};
\node at (-1,0) {$g_1$};
\node at (1,0) {$g_2$};
\node at (-1.5,1) {$g_3$};
\node at (-0.5,1) {$g_4$};
\node at (0.5,1) {$g_5$};
\node at (1.5,1) {$g_6$};
\node at (0,2) {$g_7$};
\draw (-1.125,0.25) -- (-1.375,0.75);
\draw (-0.875,0.25) -- (-0.625,0.75);
\draw (0.875,0.25) -- (0.625,0.75);
\draw (1.125,0.25) -- (1.375,0.75);
\draw (-1.25,1.25) -- (-0.25,1.75);
\draw (-0.375,1.25) -- (-0.125,1.75);
\draw (0.375,1.25) -- (0.125,1.75);
\draw (1.25,1.25) -- (0.25,1.75);
\draw (-0.25,-0.75) -- (-0.75,-0.25);
\draw (0.25,-0.75) -- (0.75,-0.25);
\end{tikzpicture}
\caption{Induced poset of the 7 representable quilts of type  $(B_{2}, B_{2})$
and its Dedekind--MacNeille  completion.  Note, $|\quilt(B_{2}, B_{2})|=16$. }
\label{fig:DMC}
\end{center}

\end{figure}
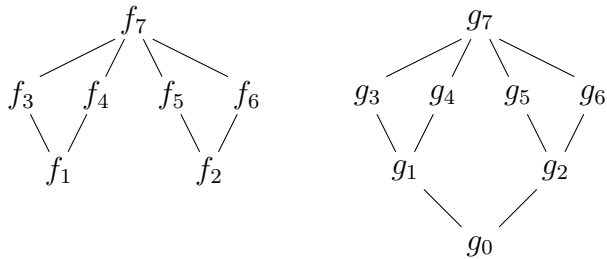

\subsection*{Irreducible quilts}

It is natural to look at the irreducible elements of a distributive lattice. It turns out that the irreducible elements of the quilt lattice, i.e.~the \emph{irreducible quilts}, have an elegant description, see~\cite{hkr}. We thank Nathan Reading for suggesting this direction of research.

\subsection*{Quilt polytopes}

There are beautiful results about the polytopes of alternating sign
matrices, matroids, and flag matroids, see e.g.~\cite{striker} and
\cite{CAMERON.DINU.MICHALEK.SEYNNAEVE}.  In 2018, Sanyal-Stump
\cite{Sanyal-Stump.2018} defined the \emph{Lipschitz polytope} of a
poset $P$, denoted $\mathcal{L}(P)$, as the set of functions $f \in
\mathbb{R}^{P}$ such that $ 0\leq f(a)\leq 1$ for all minimal elements
$a \in P$ and $0\leq f(y)-f(x)\leq 1 $ for all $x \lessdot y$ in $P$.
Therefore, the vertices of the Lipschitz polytopes are closely related
to the Dedekind maps on $P$.  This variation on Boolean growth leads
us to define $\mathcal{L}(P,Q)$ for a pair of finite ranked posets
$P,Q$ with least and greatest elements as the set of functions $f \in
\mathbb{R}^{P\times Q}$ such that
\begin{itemize}
 \item $f(x,y) = 0$ whenever $x = \hat 0_P$ or $y = \hat 0_Q$,
\item $f(\hat 1_P, \hat 1_Q) = \min\{\rank P,\  \rank
Q\}$, and
\item if $(x,y) \lessdot (x',y')$ in $P \times Q$, then $0\leq
f(x',y')-f(x,y)\leq 1$ (bounded growth).
\end{itemize}
Thus, integer lattice points of $\mathcal{L}(P,Q)$ are exactly the
quilts of type $(P,Q)$.  What is the Ehrhart polynomial for these
polytopes?  What more can be said about these
polytopes? This is the topic of the ongoing project~\cite{bk}.

\subsection*{Enumeration} As we stated in the introduction, one of the
most fascinating facts in the area is that there is a product formula
for the number of square ASMs. Corollary~\ref{cor:powerofasm} gives a
simple generalization of this statement. Is there a simple formula for
$|\quilt(P,P)|$ for some family of posets $P \neq j C_n$? Can we at
least find asymptotic formulas for $|\quilt(P_n,P_n)|$ for some nice
families of posets $P_n$, or upper and lower estimates?  Can we
improve the bounds for $|\quilt(P,B_n)|$ beyond \Cref{thm:bounds}?

\subsection*{Alternative definitions} We can generalize the definition
of a quilt slightly to account for finite ranked posets that do not
necessarily have a unique minimal and maximal element. In that case, we should replace the first two conditions in Definition~\ref{def:quilt} by
\begin{itemize}
 \item $f(x,y) = 0$ if $x$ is a minimal element of $P$ or $y$ is a minimal element of $Q$,
 \item $f(x,y) = \min\{\rank P,\ \rank Q\}$ if $x$ is a maximal element of $P$ and $y$ is a maximal element of $Q$.
\end{itemize}
Most of our results still hold, but not all; for example, there is no longer the concept of a monotone triangle with a specified top set.
Another possibility is to keep the least and greatest elements and remove the condition $f(\hat 1_P, \hat 1_Q) = \min\{\rank P,\ \rank Q\}$. One drawback of that is that elements of $\quilt(C_k,C_n)$ are no longer (in bijection with) alternating sign matrices; instead, we get \emph{partial alternating sign matrices}, see~\cite{partialASM}.

\subsection*{Standard quilts} The standard quilts
we defined in Section~\ref{sec:chain} seem worthy of further study, as
they generalize shifted standard Young tableau of shape
$(k,k-1,\ldots,1)$ and determine the asymptotics of $|\quilt(P, C_n)|$. Is there a simple way to count standard quilts
(generalizing the hook-length formula) or is that a \#P-complete
problem like computing $|\quilt(P, C_1)|$? Since determining the asymptotics of antichain quilts is \#P-complete by \Cref{thm:antichain} and since antichain quilts are typically simpler than chain quilts, we would assume that enumerating standard quilts is also a \#P-hard problem.

\subsection*{Monotone triangles} In Section~\ref{sec:chain}, we proved
that $m_P(a_1,\ldots,a_k)=|\MT_P(a_1,\ldots,a_k)|$ is a polynomial
function of $a_{1},\dotsc ,a_{k}$. For $P = C_k$, the crucial results
are the \emph{operator formula}, expressing the number of monotone
triangles via the number of Gelfand--Tsetlin patterns, and the
\emph{rotation formula}, which states that
$m_{C_k}(a_2,\ldots,a_k,a_1-k)| = (-1)^{k-1} m_{C_k}(a_1,\ldots,a_k)$,
see~\cite[Theorem 1]{MTs} and~\cite[Lemma 5]{Fischer.2005}. Is there a
way to generalize these results to arbitrary (or at least some)
posets?

\subsection*{Generalizing ASMs}\label{sub:generalizingASMs}

The literature on permutations and alternating sign matrices provide a
rich source of problems for quilts, some of which are mentioned in the
introduction and \Cref{rem:TerwilligerPoset}.  Following
Hamaker--Reiner \cite{hamakerreiner}, what are the descents for
quilts?  Is there a Hopf algebra
interpretation for quilts and an analog of the shuffle product?  Does shellability hold for the appropriately defined poset of Dedekind maps?  See~\cite{hk} for some recent progress in the direction of shellability. See
also the work of Cheballah--Giraudo--Maurice, who defined a Hopf algebra
with basis given by alternating sign matrices \cite{CGM}.

We could also consider generalizing Terwilliger's extension of the
Boolean lattice to include interlacing sets not just of the type in
\eqref{eq:interlacing.1}, but also to include
\eqref{eq:interlacing.2}.  What can be said about the interlacing
Boolean lattice with both types of interlacing conditions?

\section*{Appendix A: Computational Results}

\allowdisplaybreaks

The next three equalities illustrate Theorems~\ref{thm:antichain} and~\ref{thm:chainenumeration}. The formula for $|\quilt(B_4,C_n)|$ was actually produced using Theorem~\ref{thm:graph}.

\begin{multline*}
  \scriptstyle
  |\quilt(B_4,A_2(j))| = 2 \cdot 16^j + 12 \cdot 20^j + 6 \cdot 25^j + 24 \cdot 26^j + 8 \cdot 27^j + 24 \cdot 34^j + 8 \cdot 35^j + 14 \cdot 36^j + 8 \cdot 38^j  + 24 \cdot 39^j + 24 \cdot 42^j + 6 \cdot   47^j + 24 \cdot 49^j \\
  \scriptstyle
  + 12 \cdot 50^j + 24 \cdot 52^j + 24 \cdot 55^j + 24 \cdot 59^j + 12 \cdot 61^j + 49 \cdot 64^j + 24 \cdot 70^j + 20 \cdot 72^j + 24 \cdot 77^j + 12 \cdot   80^j + 4 \cdot 81^j + 12 \cdot 82^j + 12 \cdot 83^j + 24 \cdot 90^j + 24 \cdot 91^j \\
  \scriptstyle
  + 8 \cdot 95^j + 6 \cdot 100^j + 24 \cdot 101^j + 6 \cdot 102^j + 8 \cdot 103^j + 24 \cdot 104^j + 2 \cdot 113^j + 24 \cdot 114^j + 24 \cdot 115^j
  + 8 \cdot 122^j + 12 \cdot 128^j + 4 \cdot 129^j + 12 \cdot 133^j + 8 \cdot 147^j + 166^j
 \end{multline*}

 \begin{multline*}
\scriptstyle
 \!\!\!|\quilt(B_3,C_n)| = 1344 \binom n{12}+10080 \binom n{11}+33444 \binom n{10} +64506 \binom n{9} +79788 \binom n{8} +65652 \binom n{7}+35876 \binom n{6}+12471 \binom n{5}+2456 \binom n{4}+199 \binom n{3} \\
 \scriptstyle
 \text{for } n \geq 3
 \end{multline*}

 \begin{multline*}
  \scriptstyle
  |\quilt(B_4,C_n)| = 10651644896477184 \binom n {32} + 197055430584827904 \binom n {31} + 1738665057137541120 \binom n {30} + 9735818288500039680 \binom n {29} \\
  \scriptstyle
  + 38839556977856928768 \binom n {28}  + 117471942156471614976 \binom n {27} + 279881902757513059200 \binom n {26} + 538793272789014417984 \binom n {25}\\
  \scriptstyle
  + 852913906502788631808 \binom n {24} + 1124093660783042183328 \binom n {23} + 1244204557392229952160 \binom n {22} + 1163423387552452501296 \binom n {21} \\
  \scriptstyle
  + 922421269447363713000 \binom n {20} + 621185943976110723780 \binom n {19} + 355315109292664467516 \binom n {18} + 172335637248751133958 \binom n {17} \\
  \scriptstyle
  + 70636458716011510126 \binom n {16} + 24338243155860965610 \binom n {15} + 6997548154002120846 \binom n {14} + 1662187981311784640 \binom n {13} \\
  \scriptstyle
  + 321944626547285880 \binom n {12} + 49970302238834940 \binom n {11} + 6073377257995792 \binom n {10} + 560131126345528 \binom n {9} \\
  \scriptstyle
  + 37512372358044 \binom n {8} + 1710540931365 \binom n {7} + 48063694812 \binom n {6} + 703244285 \binom n {5} + 3813042 \binom n {4} \qquad \text{for } n \geq 4
 \end{multline*}

 The following is a list of the numbers of fundamental quilts for
 $B_3$ with a given top set given in reverse lexicographic order:

 $$ \begin{array}{|c|c||c|c||c|c||c|c||c|c||c|c||c|c|} \hline
    \scriptstyle 1, 2, 3 & \scriptstyle 199 & \scriptstyle 1, 2, 4 & \scriptstyle 1228 & \scriptstyle 1, 3, 4 & \scriptstyle 1228 & \scriptstyle 1, 2, 5 & \scriptstyle 3428 & \scriptstyle  1, 3, 5 & \scriptstyle 5615 & \scriptstyle 1, 4, 5 & \scriptstyle 3428 & \scriptstyle 1, 2, 6 & \scriptstyle 5175 \\ \hline
    \scriptstyle 1, 3, 6 & \scriptstyle 12763 & \scriptstyle 1, 4, 6 & \scriptstyle 12763 & \scriptstyle 1, 5, 6 & \scriptstyle 5175 & \scriptstyle 1, 2, 7 & \scriptstyle 4416 & \scriptstyle 1, 3, 7 & \scriptstyle 16518 & \scriptstyle 1, 4, 7 & \scriptstyle 23784 & \scriptstyle 1, 5, 7 & \scriptstyle 16518 \\ \hline
    \scriptstyle 1, 6, 7 & \scriptstyle 4416 & \scriptstyle 1, 2, 8 & \scriptstyle 2016 & \scriptstyle 1, 3, 8 & \scriptstyle 12501 & \scriptstyle 1, 4, 8 & \scriptstyle 25377 & \scriptstyle 1, 5, 8 & \scriptstyle 25377 & \scriptstyle 1, 6, 8 & \scriptstyle 12501 & \scriptstyle 1, 7, 8 & \scriptstyle 2016 \\ \hline
    \scriptstyle 1, 2, 9 & \scriptstyle 384 & \scriptstyle 1, 3, 9 & \scriptstyle 5184 & \scriptstyle 1, 4, 9 & \scriptstyle 16038 & \scriptstyle 1, 5, 9 & \scriptstyle 21294 & \scriptstyle 1, 6, 9 & \scriptstyle 16038 & \scriptstyle 1, 7, 9 & \scriptstyle 5184 & \scriptstyle 1, 8, 9 & \scriptstyle 384 \\ \hline
    \scriptstyle 1, 3, 10 & \scriptstyle 912 & \scriptstyle 1, 4, 10 & \scriptstyle 5664 & \scriptstyle 1, 5, 10 & \scriptstyle 10146 & \scriptstyle 1, 6, 10 & \scriptstyle 10146 & \scriptstyle 1, 7, 10 & \scriptstyle 5664 & \scriptstyle 1, 8, 10 & \scriptstyle 912 & \scriptstyle 1, 4, 11 & \scriptstyle 864 \\ \hline
    \scriptstyle 1, 5, 11 & \scriptstyle 2640 & \scriptstyle  1, 6, 11 & \scriptstyle 3072 & \scriptstyle 1, 7, 11 & \scriptstyle 2640 & \scriptstyle 1, 8, 11 & \scriptstyle 864 & \scriptstyle 1, 5, 12 & \scriptstyle 288 & \scriptstyle 1, 6, 12 & \scriptstyle 384 & \scriptstyle 1, 7, 12 & \scriptstyle 384 \\ \hline
    \scriptstyle 1, 8, 12 & \scriptstyle 288 \\ \cline{1-2}
 \end{array}$$

 Note that the sum of the numbers of fundamental quilts with last element of the top set equal to $m$ is equal to the coefficient of $\binom n m$ in the formula for $|\quilt(B_3,C_n)|$, e.g.\ $3428+5615+3428 = 12471$. Using this table, we can compute $|\MT_{B_3}(a_1,a_2,a_3)|$ using Theorem~\ref{thm:mt}:

 \begin{multline*} \scriptstyle
    |\MT_{B_3}(a_1,a_2,a_3)| = 199 + 1228 (a_3-a_2-1) + 1228 (a_2-a_1-1) + \dots + 384 \binom{a_2-a_1-1}5 \binom{a_3-a_2-1}4 + 288 \binom{a_2-a_1-1}6 \binom{a_3-a_2-1}3 \\
    \scriptstyle
    = \frac{1}{15} a_{1} a_{2}^8-\frac{1}{15} a_{2}^8 a_{3}-\frac{2 a_{2}^8}{15}-\frac{4}{15} a_{1}^2 a_{2}^7+\frac{4}{15} a_{2}^7 a_{3}^2+\frac{8}{15} a_{1} a_{2}^7+\frac{8}{15} a_{2}^7 a_{3} +\frac{7}{15} a_{1}^3 a_{2}^6-\frac{7}{15} a_{2}^6 a_{3}^3-\frac{14}{15} a_{1}^2 a_{2}^6-\frac{7}{15} a_{1} a_{2}^6 a_{3}^2\\
    \scriptstyle
    -\frac{14}{15} a_{2}^6 a_{3}^2-\frac{3}{40} a_{1} a_{2}^6 +\frac{7}{15} a_{1}^2 a_{2}^6 a_{3}-\frac{28}{15} a_{1} a_{2}^6 a_{3}+\frac{3}{40} a_{2}^6 a_{3} +\frac{3 a_{2}^6}{20}-\frac{7}{15} a_{1}^4 a_{2}^5+\frac{7}{15} a_{2}^5 a_{3}^4+\frac{14}{15} a_{1}^3 a_{2}^5+\frac{14}{15} a_{1} a_{2}^5 a_{3}^3+\frac{14}{15} a_{2}^5 a_{3}^3+\frac{9}{40} a_{1}^2 a_{2}^5\\
    \scriptstyle
    +\frac{14}{5} a_{1} a_{2}^5 a_{3}^2-\frac{9}{40} a_{2}^5 a_{3}^2-\frac{9}{20} a_{1} a_{2}^5 -\frac{14}{15} a_{1}^3 a_{2}^5 a_{3}+\frac{14}{5} a_{1}^2 a_{2}^5 a_{3}-\frac{9}{20} a_{2}^5 a_{3}+\frac{4}{15} a_{1}^5 a_{2}^4-\frac{4}{15} a_{2}^4 a_{3}^5-\frac{1}{3} a_{1}^4 a_{2}^4-a_{1} a_{2}^4 a_{3}^4-\frac{1}{3} a_{2}^4 a_{3}^4-\frac{31}{24} a_{1}^3 a_{2}^4\\
    \scriptstyle
    -\frac{1}{3} a_{1}^2 a_{2}^4 a_{3}^3 -\frac{10}{3} a_{1} a_{2}^4 a_{3}^3+\frac{31}{24} a_{2}^4 a_{3}^3+\frac{21}{8} a_{1}^2 a_{2}^4+\frac{1}{3} a_{1}^3 a_{2}^4 a_{3}^2-2 a_{1}^2 a_{2}^4 a_{3}^2-\frac{11}{4} a_{1} a_{2}^4 a_{3}^2+\frac{21}{8} a_{2}^4 a_{3}^2-\frac{13}{12} a_{1} a_{2}^4+a_{1}^4 a_{2}^4 a_{3} -\frac{10}{3} a_{1}^3 a_{2}^4 a_{3}+\frac{11}{4} a_{1}^2 a_{2}^4 a_{3}\\
    \scriptstyle
    -3 a_{1} a_{2}^4 a_{3}+\frac{13}{12} a_{2}^4 a_{3}-\frac{6 a_{2}^4}{5}-\frac{1}{15} a_{1}^6 a_{2}^3+\frac{1}{15} a_{2}^3 a_{3}^6-\frac{4}{15} a_{1}^5 a_{2}^3+\frac{2}{3} a_{1} a_{2}^3 a_{3}^5 -\frac{4}{15} a_{2}^3 a_{3}^5+\frac{53}{24} a_{1}^4 a_{2}^3+\frac{1}{3} a_{1}^2 a_{2}^3 a_{3}^4+\frac{8}{3} a_{1} a_{2}^3 a_{3}^4 -\frac{53}{24} a_{2}^3 a_{3}^4\\
    \scriptstyle
    -\frac{9}{2} a_{1}^3 a_{2}^3+\frac{4}{3} a_{1}^2 a_{2}^3 a_{3}^3+\frac{11}{3} a_{1} a_{2}^3 a_{3}^3-\frac{9}{2} a_{2}^3 a_{3}^3 +\frac{13}{6} a_{1}^2 a_{2}^3-\frac{1}{3} a_{1}^4 a_{2}^3 a_{3}^2+\frac{4}{3} a_{1}^3 a_{2}^3 a_{3}^2+3 a_{1} a_{2}^3 a_{3}^2-\frac{13}{6} a_{2}^3 a_{3}^2+\frac{12}{5} a_{1} a_{2}^3-\frac{2}{3} a_{1}^5 a_{2}^3 a_{3}+\frac{8}{3} a_{1}^4 a_{2}^3 a_{3}-\\
    \scriptstyle
    \frac{11}{3} a_{1}^3 a_{2}^3 a_{3} +3 a_{1}^2 a_{2}^3 a_{3}+\frac{12}{5} a_{2}^3 a_{3}+\frac{1}{5} a_{1}^6 a_{2}^2-\frac{1}{5} a_{1} a_{2}^2 a_{3}^6+\frac{1}{5} a_{2}^2 a_{3}^6-\frac{13}{15} a_{1}^5 a_{2}^2-\frac{2}{5} a_{1}^2 a_{2}^2 a_{3}^5-\frac{2}{5} a_{1} a_{2}^2 a_{3}^5+\frac{13}{15} a_{2}^2 a_{3}^5 +\frac{17}{24} a_{1}^4 a_{2}^2+\frac{1}{3} a_{1}^3 a_{2}^2 a_{3}^4\\
    \scriptstyle
    -3 a_{1}^2 a_{2}^2 a_{3}^4+\frac{55}{24} a_{1} a_{2}^2 a_{3}^4+\frac{17}{24} a_{2}^2 a_{3}^4+\frac{133}{40} a_{1}^3 a_{2}^2-\frac{1}{3} a_{1}^4 a_{2}^2 a_{3}^3+\frac{8}{3} a_{1}^3 a_{2}^2 a_{3}^3 -\frac{121}{12} a_{1}^2 a_{2}^2 a_{3}^3+\frac{32}{3} a_{1} a_{2}^2 a_{3}^3-\frac{133}{40} a_{2}^2 a_{3}^3-\frac{361}{60} a_{1}^2 a_{2}^2+\frac{2}{5} a_{1}^5 a_{2}^2 a_{3}^2\\
    \scriptstyle
    -3 a_{1}^4 a_{2}^2 a_{3}^2+\frac{121}{12} a_{1}^3 a_{2}^2 a_{3}^2-\frac{41}{2} a_{1}^2 a_{2}^2 a_{3}^2 +\frac{659}{40} a_{1} a_{2}^2 a_{3}^2-\frac{361}{60} a_{2}^2 a_{3}^2-\frac{127}{120} a_{1} a_{2}^2+\frac{1}{5} a_{1}^6 a_{2}^2 a_{3}-\frac{2}{5} a_{1}^5 a_{2}^2 a_{3}-\frac{55}{24} a_{1}^4 a_{2}^2 a_{3}+\frac{32}{3} a_{1}^3 a_{2}^2 a_{3}-\frac{659}{40} a_{1}^2 a_{2}^2 a_{3}\\
    \scriptstyle
  +\frac{29}{6} a_{1} a_{2}^2 a_{3}+\frac{127}{120} a_{2}^2 a_{3}+\frac{11 a_{2}^2}{60}-\frac{1}{5} a_{1}^6 a_{2}+\frac{1}{5} a_{1}^2 a_{3}^6
   a_{2}-\frac{2}{5} a_{1} a_{2} a_{3}^6+\frac{1}{5} a_{2} a_{3}^6+\frac{22}{15} a_{1}^5 a_{2}-\frac{2}{15} a_{1}^3 a_{2} a_{3}^5 +\frac{8}{5} a_{1}^2 a_{2} a_{3}^5-\frac{44}{15} a_{1} a_{2} a_{3}^5+\frac{22}{15} a_{2} a_{3}^5\\
   \scriptstyle
   -\frac{529}{120} a_{1}^4 a_{2}-\frac{2}{3} a_{1}^3 a_{2} a_{3}^4+\frac{121}{24} a_{1}^2 a_{2} a_{3}^4-\frac{35}{4} a_{1} a_{2} a_{3}^4+\frac{529}{120} a_{2} a_{3}^4 +\frac{289}{60} a_{1}^3 a_{2}+\frac{2}{15} a_{1}^5 a_{2} a_{3}^3-\frac{2}{3} a_{1}^4 a_{2} a_{3}^3+\frac{41}{6} a_{1}^2 a_{2} a_{3}^3-\frac{659}{60} a_{1} a_{2} a_{3}^3+\frac{289}{60} a_{2} a_{3}^3\\
   \scriptstyle +\frac{127}{120} a_{1}^2 a_{2}-\frac{1}{5} a_{1}^6 a_{2} a_{3}^2 +\frac{8}{5} a_{1}^5 a_{2} a_{3}^2-\frac{121}{24} a_{1}^4 a_{2} a_{3}^2+\frac{41}{6} a_{1}^3 a_{2} a_{3}^2-\frac{29}{12} a_{1} a_{2} a_{3}^2-\frac{127}{120} a_{2} a_{3}^2-\frac{11}{60} a_{1} a_{2}-\frac{2}{5} a_{1}^6 a_{2} a_{3}+\frac{44}{15} a_{1}^5 a_{2} a_{3} -\frac{35}{4} a_{1}^4 a_{2} a_{3}\\
   \scriptstyle+\frac{659}{60} a_{1}^3 a_{2} a_{3}-\frac{29}{12} a_{1}^2 a_{2} a_{3}-\frac{11}{60} a_{2} a_{3}+\frac{a_{1}^6}{15}-\frac{1}{15} a_{1}^3 a_{3}^6+\frac{1}{5} a_{1}^2 a_{3}^6-\frac{1}{5} a_{1} a_{3}^6+\frac{a_{3}^6}{15}-\frac{3 a_{1}^5}{5} +\frac{2}{15} a_{1}^4 a_{3}^5-\frac{14}{15} a_{1}^3 a_{3}^5+\frac{31}{15} a_{1}^2 a_{3}^5-\frac{28}{15} a_{1} a_{3}^5+\frac{3 a_{3}^5}{5}\\
   \scriptstyle
   +\frac{55 a_{1}^4}{24}-\frac{2}{15} a_{1}^5 a_{3}^4+\frac{4}{3} a_{1}^4 a_{3}^4-\frac{41}{8} a_{1}^3 a_{3}^4+\frac{217}{24} a_{1}^2 a_{3}^4 -\frac{889}{120} a_{1} a_{3}^4+\frac{55 a_{3}^4}{24}-\frac{15 a_{1}^3}{4}+\frac{1}{15} a_{1}^6 a_{3}^3-\frac{14}{15} a_{1}^5 a_{3}^3+\frac{41}{8} a_{1}^4 a_{3}^3-\frac{43}{3} a_{1}^3 a_{3}^3+\frac{2437}{120} a_{1}^2 a_{3}^3\\
   \scriptstyle
   -\frac{839}{60} a_{1} a_{3}^3 +\frac{15 a_{3}^3}{4}+\frac{137 a_{1}^2}{120}+\frac{1}{5} a_{1}^6 a_{3}^2-\frac{31}{15} a_{1}^5 a_{3}^2+\frac{217}{24} a_{1}^4 a_{3}^2-\frac{2437}{120} a_{1}^3 a_{3}^2+\frac{1331}{60} a_{1}^2 a_{3}^2-\frac{1223}{120} a_{1} a_{3}^2+\frac{137 a_{3}^2}{120} -\frac{3 a_{1}}{20}+\frac{1}{5} a_{1}^6 a_{3}\\
   \scriptstyle -\frac{28}{15} a_{1}^5 a_{3}+\frac{889}{120} a_{1}^4 a_{3}-\frac{839}{60} a_{1}^3 a_{3}+\frac{1223}{120} a_{1}^2 a_{3}-\frac{21}{10} a_{1} a_{3}+\frac{3 a_{3}}{20}.
 \end{multline*}

\section*{Appendix B: Numerical Sequences}\label{sec:oeis}

\subsection{Generalized Dedekind Numbers}

From \Cref{sec:def}, $d_{k}(B_{n})$ for $0\leq k\leq n\leq 5$ are
given by

$$\begin{array}{c|cccccc}
   n \backslash k & 0 & 1 & 2 & 3 & 4 & 5  \\ \hline
   0 & 1 & 0 & 0 & 0 & 0 & 0 \\
   1 & 1 & 1 & 0 & 0 & 0 & 0 \\
   2 & 1 & 4 & 1 & 0 & 0 & 0 \\
   3 & 1 & 18 & 18 & 1 & 0 & 0 \\
   4 & 1 & 166 & 656 & 166 & 1 & 0 \\
   5 & 1 & 7579 & 189967 & 189967 & 7579 & 1 \end{array}
$$
Note OEIS A007153 appears in column 1.  Reading the triangle of
nonzero entries by rows from the top we have the sequence 1, 1, 1, 1, 4,
1, 1, 18, 18, 1, 1, 166, 656, 166, 1, 1, 7579, 189967, 189967, 7579, 1.
Note the symmetry naturally comes from complementing sets and values.
This sequence is \cite[A374819]{oeis}.



\subsection{Boolean-Chain numbers}

The square table of numbers $|\quilt(B_{n}, C_{k})|$ for $1\leq
n\leq 4$ and $1\leq k\leq 7$ plus for $n=5,k=1,2$ and $n=6,k=1$ are given by



$$\begin{array}{c|ccccccc}
   n \backslash k & 1 & 2 & 3 & 4 & 5 & 6 & 7  \\ \hline
                    1 & 1 & 2 & 3 & 4 & 5 & 6 & 7 \\
                    2 & 4 & 4 & 17 & 46 & 100 & 190 & 329 \\
                    3 & 18 & 199 & 199 & 3252 & 26741 & 151522 & 671600 \\
                    4 & 166 & 47000 & 3813042 & 3813042 & 722309495 & 52340356152 & 2061888381504\\
	              5 & 7579 & 410131245  & ?  &  ? & ? & ? & ?\\
   	              6 & 7828352  & ?  & ?  &  ? & ? & ? & ?
\end{array}$$
Reading antidiagonals starting at $k=n=1$, we have 1, 2, 4, 3, 4, 18,
4, 17, 199, 166, 5, 46, 199, 47000, 7579, 6, 100, 3252, 3813042, 410131245,
7828352.  This sequence is \cite[A374820]{oeis}.


\subsection{Antichain-Boolean numbers}

The table of $|\quilt(A_2(j), B_{n})|$ for $1\leq n\leq 4$ and
$1\leq j \leq 6$, plus $|\quilt(A_{2}(1), B_{5})|$ is given by

$$\begin{array}{c|cccccc}
   n \backslash j & 1 & 2 & 3 & 4 & 5 & 6\\ \hline
                1 & 2 & 4 & 8 & 16 & 32 & 64\\
2 & 4 & 16 & 64 & 256 & 1024 & 4096 \\
3 & 199 & 2309 & 28225 & 364217 & 4960009 & 71091689 \\
4 & 47000 & 4001278 & 384285926 & 40139162386 & 4455115959710 & 517943027803618 \\
5 & 410131245  & ?  & ? & ? & ? & ?\\
\end{array}$$
Reading antidiagonals starting at $k=n=1$, we have 2, 4, 4, 8, 16,
199, 16, 64, 2309, 47000, 32, 256, 28225, 4001278, 410131245, 64,
1024, 364217, 384285926.  This sequence is \cite[A374821]{oeis}.


\subsection{Antichain-Chain Quilt Numbers}

The number of quilts of type $(A_2(j), C_k)$, where $j$ is the
column index for $1\leq j\leq 8$ and $k$ is the row index for $1\leq
k\leq 8$ is given by the table 
\begin{small}
$$\begin{array}{c|cccccccc}
   k \backslash j & 1 & 2 & 3 & 4 & 5 & 6 & 7 & 8\\ \hline
1 & 2 & 4 & 8 & 16 & 32 & 64 & 128 & 256 \\
2 & 2 & 4 & 8 & 16 & 32 & 64 & 128 & 256 \\
3 & 7 & 17 & 43 & 113 & 307 & 857 & 2443 & 7073 \\
4 & 16 & 46 & 142 & 466 & 1606 & 5746 & 21142 & 79426 \\
5 & 30 & 100 & 366 & 1444 & 6030 & 26260 & 117966 & 542404 \\
6 & 50 & 190 & 806 & 3718 & 18230 & 93430 & 494726 & 2684998 \\
7 & 77 & 329 & 1589 & 8393 & 47237 & 278249 & 1695029 & 10592393 \\
8 & 112 & 532 & 2884 & 17164 & 109012 & 725212 & 4992484 & 35277004
\end{array}$$
\end{small}
Reading antidiagonals we have 2, 4, 2, 8, 4, 7, 16, 8, 17, 16, etc.  This sequence is \cite[A374822]{oeis}.

\subsection{Chain-Chain Numbers }

The numbers of quilts of type $(C_k, C_n)$ is also the number of
rectangular alternating sign matrices in $\ASM_{k,n}$. This sequence
is included \cite[A297622]{oeis}, where they include the cases where
$k=0$.    Note, $|\quilt(C_0,
C_n)|=1$ for all $n\geq 0$.  The numbers $|\quilt(C_k, C_n)|$
for $1\leq k\leq 10$ are



$$\begin{array}{c|cccccccccc}
   n \backslash k & 1 & 2 & 3 & 4 & 5 & 6 & 7 & 8 & 9 & 10 \\ \hline
1 & 1 & 2 & 3 & 4 & 5 & 6 & 7 & 8 & 9 & 10 \\
2 & {} & 2 & 7 & 16 & 30 & 50 & 77 & 112 & 156 & 210 \\
3 & {} & {} & 7 & 42 & 149 & 406 & 938 & 1932 & 3654 & 6468 \\
4 & {} & {} & {} & 42 & 429 & 2394 & 9698 & 31920 & 90576 & 229680 \\
5 & {} & {} & {} & {} & 429 & 7436 & 65910 & 403572 & 1931325 & 7722110 \\
6 & {} & {} & {} & {} & {} & 7436 & 218348 & 3096496 & 29020904 & 205140540 \\
7 & {} & {} & {} & {} & {} & {} & 218348 & 10850216 & 247587252 & 3586953760 \\
8 & {} & {} & {} & {} & {} & {} & {} & 10850216 & 911835460 & 33631201864 \\
\end{array}
$$
Reading down columns for the triangle of numbers we have 1, 2, 2, 3,
7, 7, 4, 16, 42, 42, 5, 30, 149, 429, 429, 6, 50, 406, 2394, 7436,
7436, \dots .  Note, \cite[A005130]{oeis} is the sequence counting the
number of square ASMs.  It starts out 1, 1, 2, 7, 42, 429, 7436,
218348, 10850216, as shown in the diagonal.


\subsection{Antichain-Antichain Quilt Numbers}

The quilts of type $(A_2(j), A_2(k))$ are in bijection with the
$j\times k$ binary arrays, so the formula is $2^{jk}$ for all $j,k\geq
1$.

\subsection{Boolean-Boolean Quilt Numbers}

The triangular array of the number of ASM quilts of type $(B_n,
B_{k})$ begins with

$$\begin{array}{c|ccccc}
   n \backslash k & 1 & 2 & 3  & 4   & 5 \\ \hline
                1 & 1 & 4 & 18 & 166 & 7579 \\
2& {} & 16 & 2309 & 4001278  & {?}\\
3& {} & {} & 2406862  & {?} & {?}\\
\end{array}
$$
Reading down columns we have 1, 4, 16, 18, 2309, 2406862, 166,
4001278. This sequence is \cite[A374824]{oeis}.

\section*{Funding}

M.~Konvalinka was partially supported by ERC AdG KARST and by projects and programs J1-2452, N1-0218, P1-0294, P1-0297 of the Slovenian Research Agency.

\section*{Acknowledgments}

We would like to thank Anders Claesson, Herman Chau, Ilse Fischer, Hans
H\"ongesberg, Nathan Reading, Stark Ryan, Raman Sanyal, Jessica Striker, and Joshua
Swanson for helpful conversations related to this
work.  We would also like to thank an anonymous referee for carefully reading the paper and for their thoughtful suggestions.

\bibliographystyle{amsalpha}
\newcommand{\etalchar}[1]{$^{#1}$}
\providecommand{\bysame}{\leavevmode\hbox to3em{\hrulefill}\thinspace}
\providecommand{\MR}{\relax\ifhmode\unskip\space\fi MR }
\providecommand{\MRhref}[2]{%
  \href{http://www.ams.org/mathscinet-getitem?mr=#1}{#2}
}
\providecommand{\href}[2]{#2}

\end{document}